\documentclass[12pt]{amsart}

\usepackage{etex} 
\usepackage[latin1]{inputenc}	
\usepackage[T1]{fontenc}
\usepackage[english]{babel}
\usepackage{amsmath,amssymb}
\usepackage{wasysym}
\usepackage{stmaryrd}
\usepackage[table]{xcolor}
\usepackage{color}
\usepackage{dsfont}\let\mathbb\mathds
\usepackage[all]{xy}
\usepackage{graphicx}
\usepackage{mathenv}
\usepackage{lastpage}
\usepackage{fancyhdr}	
\usepackage{subfigure}
\usepackage{geometry}
\usepackage[version=3]{mhchem}
\usepackage[force,almostfull]{textcomp}
\usepackage{array}
\usepackage{lipsum}		
\usepackage{eso-pic}
\usepackage{multirow}
\usepackage{fancybox}
\usepackage{cite}
\usepackage{amsthm}
\usepackage{listings}
\usepackage{lscape}
\usepackage{multicol}		
\usepackage{setspace}
\usepackage{times}
\usepackage{eurosym}
\usepackage{latexsym}
\usepackage{ulem}
\usepackage{lmodern}
\usepackage{colortbl}
\usepackage{rotating}
\usepackage{type1cm}
\usepackage{lettrine}
\usepackage{ragged2e}
\usepackage{mathrsfs}
\usepackage{scrtime}
\usepackage{enumitem}

\usepackage{longtable}
\usepackage{url}
\usepackage{nomentbl}
\usepackage{nomencl}
\usepackage{hyperref}
\hypersetup{
pdfpagemode=UseOutlines,      
pdfstartview=Fit,             
pdffitwindow=true,            
pdfpagelayout=TwoColumnsRight,
pdftoolbar=true,              
pdfmenubar=true,              
bookmarksopen=false,          
bookmarksnumbered=true,       
colorlinks=true,              
pdfauthor={Ton nom},          
pdftitle={Titre PDF},         
pdfcreator=PDFLaTeX,          %
pdfproducer=PDFLaTeX,         %
linkcolor=blue,               
urlcolor=blue,                
anchorcolor=black,            
citecolor=blue,               
frenchlinks=true,             
pdfborder={0 0 0}             
}
\usepackage{pgf, tikz}
\usetikzlibrary{matrix,positioning}
\usetikzlibrary{arrows,decorations.markings}

\newtheorem{theorem}{Theorem}[section]

\newtheorem{definition}[theorem]{Definition}
\newtheorem{proposition}[theorem]{Proposition}
\newtheorem{lemma}[theorem]{Lemma}
\newtheorem{corollary}[theorem]{Corollary}

\newtheorem{remark}[theorem]{Remark}

\def\Card{\operatorname{Card}}
\def\MCard{\operatorname{MCard}}

\addto\captionsfrenchb{}
\addto\captionsfrenchb{}

    \newlength{\myarrowsize} 
    \newlength{\myoldlinewidth}

    \pgfarrowsdeclare{myto}{myto}{
        \pgfsetdash{}{0pt} 
        \pgfsetbeveljoin 
        \pgfsetroundcap 
        \setlength{\myarrowsize}{0.5pt}
        \addtolength{\myarrowsize}{.5\pgflinewidth}
        \pgfarrowsleftextend{-4\myarrowsize-.5\pgflinewidth} 
        \pgfarrowsrightextend{.7\pgflinewidth}
    }{
        \setlength{\myarrowsize}{0.4pt} 
        \addtolength{\myarrowsize}{.3\pgflinewidth}  
        \setlength{\myoldlinewidth}{\pgflinewidth}
        \pgfsetroundjoin
        \pgfsetlinewidth{0.0001pt}
        \pgfpathmoveto{\pgfpoint{0.43\myarrowsize}{0}}
        \pgfpatharc{0}{70}{0.14\myarrowsize}
        \pgfpatharc{-80}{-169.5}{4\myarrowsize}
        \pgfpatharc{150}{189}{0.95\myarrowsize and 0.95\myarrowsize}
        \pgfpatharc{0}{-40}{15\myarrowsize}
        \pgfpathmoveto{\pgfpoint{0.43\myarrowsize}{0}}
        \pgfpatharc{0}{-70}{0.14\myarrowsize}
        \pgfpatharc{80}{169.5}{4\myarrowsize}
        \pgfpatharc{-150}{-189}{0.95\myarrowsize and 0.95\myarrowsize}
        \pgfpatharc{0}{40}{15\myarrowsize}
        \pgfpathclose
        \pgfsetstrokeopacity{0.25}
        \pgfusepathqfillstroke
    }

    \pgfarrowsdeclare{myonto}{myonto}{
        \pgfsetdash{}{0pt} 
        \pgfsetbeveljoin 
        \pgfsetroundcap 
        \setlength{\myarrowsize}{0.5pt}
        \addtolength{\myarrowsize}{.5\pgflinewidth}
        \pgfarrowsleftextend{-4\myarrowsize-.5\pgflinewidth} 
        \pgfarrowsrightextend{.7\pgflinewidth}
    }{
        \setlength{\myarrowsize}{0.4pt} 
        \addtolength{\myarrowsize}{.3\pgflinewidth}  
        \setlength{\myoldlinewidth}{\pgflinewidth}
        \pgfsetroundjoin
        \pgfsetlinewidth{0.0001pt}
        \pgfpathmoveto{\pgfpoint{0.43\myarrowsize}{0}}
        \pgfpatharc{0}{70}{0.14\myarrowsize}
        \pgfpatharc{-80}{-169.5}{4\myarrowsize}
        \pgfpatharc{150}{189}{0.95\myarrowsize and 0.95\myarrowsize}
        \pgfpatharc{0}{-40}{15\myarrowsize}
        \pgfpathmoveto{\pgfpoint{-7\myarrowsize}{0}}
				\pgfpatharc{0}{70}{0.14\myarrowsize}
        \pgfpatharc{-80}{-169.5}{4\myarrowsize}
        \pgfpatharc{150}{189}{0.95\myarrowsize and 0.95\myarrowsize}
        \pgfpatharc{0}{-40}{15\myarrowsize}
        \pgfpathmoveto{\pgfpoint{0.43\myarrowsize}{0}}
        \pgfpatharc{0}{-70}{0.14\myarrowsize}
        \pgfpatharc{80}{169.5}{4\myarrowsize}
        \pgfpatharc{-150}{-189}{0.95\myarrowsize and 0.95\myarrowsize}
        \pgfpatharc{0}{40}{15\myarrowsize}
			  \pgfpathmoveto{\pgfpoint{-7\myarrowsize}{0}}
        \pgfpatharc{0}{-70}{0.14\myarrowsize}
        \pgfpatharc{80}{169.5}{4\myarrowsize}
        \pgfpatharc{-150}{-189}{0.95\myarrowsize and 0.95\myarrowsize}
        \pgfpatharc{0}{40}{15\myarrowsize}
        \pgfpathclose
        \pgfsetstrokeopacity{0.25}
        \pgfusepathqfillstroke
        \pgfusepathqfillstroke
    }

 \pgfarrowsdeclare{myhook}{myhook}{
        \setlength{\myarrowsize}{0.6pt}
        \addtolength{\myarrowsize}{.5\pgflinewidth}
        \pgfarrowsleftextend{-4\myarrowsize-.5\pgflinewidth} 
        \pgfarrowsrightextend{.7\pgflinewidth}
    }{
        \setlength{\myarrowsize}{0.6pt} 
        \addtolength{\myarrowsize}{.5\pgflinewidth}  
        \pgfsetdash{}{+0pt}
        \pgfsetroundcap
        \pgfpathmoveto{\pgfqpoint{-2pt}{-6\pgflinewidth}}
        \pgfpathcurveto
            {\pgfqpoint{4\pgflinewidth}{-4.667\pgflinewidth}}
            {\pgfqpoint{4\pgflinewidth}{0pt}}
            {\pgfpointorigin}
        \pgfusepathqstroke
    }

		\pgfarrowsdeclare{my to}{my to}
{
  \pgfarrowsleftextend{-2\pgflinewidth}
  \pgfarrowsrightextend{\pgflinewidth}
}
{
  \pgfsetlinewidth{0.8\pgflinewidth}
  \pgfsetdash{}{0pt}
  \pgfsetroundcap
  \pgfsetroundjoin
  \pgfpathmoveto{\pgfpoint{-5.5\pgflinewidth}{7.5\pgflinewidth}}
  \pgfpathcurveto
  {\pgfpoint{-4.0\pgflinewidth}{0.1\pgflinewidth}}
  {\pgfpoint{0pt}{0.25\pgflinewidth}}
  {\pgfpoint{0.75\pgflinewidth}{0pt}}
  \pgfpathcurveto
  {\pgfpoint{0pt}{-0.25\pgflinewidth}}
  {\pgfpoint{-4.0\pgflinewidth}{-0.1\pgflinewidth}}
  {\pgfpoint{-5.5\pgflinewidth}{-7.5\pgflinewidth}}
  \pgfusepathqstroke
}

		\pgfarrowsdeclare{mytodashed}{mytodashed}
{
 \pgfarrowsleftextend{-2\pgflinewidth}
 \pgfarrowsrightextend{\pgflinewidth}
}
{
 \pgfsetlinewidth{0.8\pgflinewidth}
  \pgfsetdash{{0.8cm}{0.8cm}}{0cm}
\pgfsetroundcap
\pgfsetroundjoin
  \pgfpathmoveto{\pgfpoint{-5.5\pgflinewidth}{7.5\pgflinewidth}}
\pgfpathcurveto
{\pgfpoint{-4.0\pgflinewidth}{0.1\pgflinewidth}}
 {\pgfpoint{0pt}{0.25\pgflinewidth}}
{\pgfpoint{0.75\pgflinewidth}{0pt}}
\pgfpathcurveto
{\pgfpoint{0pt}{-0.25\pgflinewidth}}
  {\pgfpoint{-4.0\pgflinewidth}{-0.1\pgflinewidth}}
 {\pgfpoint{-5.5\pgflinewidth}{-7.5\pgflinewidth}}
\pgfusepathqstroke
}

\tikzstyle{vecArrow} = [thick, decoration={markings,mark=at position
   1 with {\arrow[semithick]{open triangle 60}}},
   double distance=1.4pt, shorten >= 5.5pt,
   preaction = {decorate},
   postaction = {draw,line width=1.4pt, white,shorten >= 4.5pt}]
\tikzstyle{innerWhite} = [semithick, white,line width=1.4pt, shorten >= 4.5pt]

	\makeatletter
	\newcommand\POSITION[3]{%
	\begingroup
	\@tempdim@x=0cm
	\@tempdim@y=\paperheight
	\advance\@tempdim@x#1
	\advance\@tempdim@y-#2
	\put(\LenToUnit{\@tempdim@x},\LenToUnit{\@tempdim@y}){#3}%
	\endgroup
	}
\makeatother

\addto\captionsenglish{}

\begin{document}

 \title[Towers of fully commutative elements of type  $\tilde B$ and  $\tilde D$]{On fully commutative 
elements of type $\tilde B$ and  $\tilde D$ \\ }
\author{Sadek AL HARBAT}
\address{Instituto de Matem\'atica y F\'\i sica, Universidad de Talca, Chile} 
\email{sadekharbat@inst-mat.utalca.cl}
 \thanks{IMAFI, Universidad de Talca. \\
Supported by FONDECYT grant number 3170544.} 

	\begin{abstract}
	We define a tower of injections of  $\tilde{B}$-type (resp. $\tilde{D}$-type) Coxeter groups $W(\tilde B_{n})$ (resp. $W(\tilde D_{n})$) for $n\geq 3$. Let $W^c(\tilde B_{n})$ (resp. $W^c(\tilde D_{n})$) be the set of  fully commutative elements in $W(\tilde B_{n})$ (resp. $W(\tilde D_{n})$), we classify  the elements of this set by giving a normal form for them. We define a $\tilde{B}$-type tower of Hecke algebras and we use the faithfulness at the Coxeter level to show that this last tower is a tower of injections.
	  We use this normal form to define two injections from $W^c(\tilde B_{n-1})$ into $W^c(\tilde B_{n})$. We then define the tower of affine Temperley-Lieb algebras of type $\tilde{B }$ and use the injections above to prove the faithfulness of this tower.   We follow the same track  for $\tilde{D}$-type objects. 
	\end{abstract}

\date{\today}

 	\maketitle
	

\section{Introduction} 

A great deal of this work is a continuation of \cite{Sadek_2016} and \cite{Sadek_2017}, that is: giving a reduced expression that is a normal form for any fully commutative element in $\tilde B$-type and  $\tilde D$-type Coxeter groups, viewing them as objects in infinite towers and using this view to show the injectivity of the associated towers of affine Temperley-Lieb algebras. In the above mentioned works    $\tilde A$-type and $\tilde C$-type 
were treated.\\

Let $(W,S)$  be a Coxeter system. We say that w in  $W $ is {\it fully commutative} if any reduced expression for $ w$
can be obtained from any other using only commutation relations among the members of the set $S$. Call the set of such elements $W^c$.\\

Enumerating fully commutative elements by a normal form is an interesting subject in its own (see Stembridge \cite{St} for the three infinite families of finite Coxeter groups),  and also has important applications and consequences. For the   infinite affine families, the author  treated the $\tilde A$-type for a question coming from braid group theory, and,  encouraged by  a question of Luis Paris,   went on, in the same spirit,  to treat  the three other affine cases.  

First of all it has applications to the study of Temperley-Lieb algebras -- which are quotients of Hecke algebras -- since  fully commutative elements in a Coxeter group index at least two known  bases of the corresponding Temperley-Lieb algebra. 
In particular, as  a consequence of the normal form for affine types $\tilde A$ to $\tilde D$, we can prove the injectivity  of  associated towers of Temperley-Lieb algebras, by producing 
  injective towers of the sets of fully commutative elements. 
It can also be a decisive tool in the study of traces, particularly traces of Markov type, on those towers of  Temperley-Lieb algebras, as was done in \cite{Sadek_2013_2}. 

Other than that, by providing accessible computations of some Kazhdan-Lusztig polynomials, as was done in 
\cite{Green2009},  it can give a way to attack  the study of the 0-1 conjecture, or of a fully commutative version of this conjecture since the general formulation has proved wrong \cite{MW}. 

In yet another direction, 
 the author in  a forthcoming work is giving an explicit Coxeter-length generating function for fully commutative elements as a direct application of this enumeration, moreover an {\it affine-length} generating function (see below). This gives a clear idea about the growth of the infinite-dimensional affine  Temperley-Lieb algebras.\\

Towers are to be defined here on four levels: Coxeter groups and Hecke algebras, fully commutative elements and  Temperley-Lieb algebras; and  in two affine types  $\tilde B$ and  $\tilde D$. We prove the injectivity on the Coxeter group level,   a crucial step towards  injectivity 
on the Hecke level, which is more technical  though.  

We gave proofs of injectivity 
of the tower of affine Hecke algebras  over the field $\mathbb Q[q, q^{-1}]$ in \cite[Proposition 4.3.3]{Sadek_Thesis} for type $\tilde A$ and in 
\cite[Proposition 3.3]{Sadek_2017}  for type $\tilde C$, using specialization at $q=1$,  that maps Hecke algebras onto group algebras.  
We give here  a proof of injectivity for type   $\tilde B$ which is independent of the ring of definition and goes as well for   type $\tilde C$, hence extending the result in \cite{Sadek_2017}. But for  $\tilde D$-type the proof keeps similar to type $\tilde A$ and needs the  same condition on the ground ring.

Let us pause for this. An element in a Coxeter group is called {\it rigid} if it has only one  reduced expression. In the monomorphisms 
$W(\tilde B_n) \hookrightarrow W(\tilde B_{n+1})$ and  $W(\tilde C_n) \hookrightarrow W(\tilde C_{n+1})$ 
defined below, the Coxeter generators of the first group map to rigid elements of the second; however this does  not hold for $W(\tilde A_n) \hookrightarrow W(\tilde A_{n+1})$ nor  $W(\tilde D_n) \hookrightarrow W(\tilde D_{n+1})$. 
The author believes that our proof 
for type   $\tilde B$  is generic for any reflexion subgroup of a Coxeter group whenever the reflexions are rigid!

Precisely, recall from \cite{Deo} that a subgroup $Y$ of a Coxeter group 
$(W,S)$  generated by reflexions is itself a Coxeter group $(Y, S_Y)$ with a canonical set of reflexions $S_Y$. Paolo Sentinelli  has recently asked  whether, in this situation,  the  corresponding Hecke algebras embed accordingly. 
Actually the existence of such a morphism is not guaranteed. For instance, there is no morphism of Hecke algebras attached to the monomorphism $\beta : W(\tilde B_n) \hookrightarrow 
W(\tilde C_n) $ in~\S \ref{group tower} (see Remark \ref{braid}).   
The author believes the following: if there exists a corresponding homomorphism of Hecke algebras, then it is injective, and  
  in the case where the elements of $S_Y$ are rigid in $(W,S)$,   that it can be proved by elaborating on the  present proof in type  $\tilde B$. \\

The classification of fully commutative elements  and the proof of injectivity of the towers  of  Temperley-Lieb algebras in the four affine types are based on the notion of "affine length", this is to be explained briefly as follows: we choose to see any affine Coxeter group as an "affinization" of a finite Coxeter group in the obvious  way (but for the case of  $\tilde B$-type which can be viewed as an extension of either $B$-type or $D$-type, we choose  the latter for technical reasons). This "affinization" is simply adding a simple reflexion, say $a$, as follows:  \\

  \begin{figure}[ht]
      \begin{tikzpicture}
         \begin{scope}[scale =  0.6]
            \filldraw (0,0) circle (2pt); 
            \node at (0,-0.5) { }; 
            \draw (0,0) -- (1.5, 0);
 
            \filldraw (1.5,0) circle (2pt);
            \node at (1.5,-0.5) { };
            \draw (1.5,0) -- (2.5, 0);
            
            \node at (3.5,0) {$\dots$};
            
            \draw (4.5,0) -- (5.5, 0);
            \filldraw (5.5,0) circle (2pt);
            \node at (5.5,-0.5) { };
      
            \draw (5.5,0) -- (7, 0);
            \filldraw (7,0) circle (2pt);
            \node at (7,-0.5) { };           
            
         \end{scope}
         \begin{scope}[scale =  0.6, xshift = 14cm]
         
            \draw[line width =1.7pt][->] (-4.5,0) -- (-1.5,0);
            \filldraw (0,0) circle (2pt); 
            \node at (0,-0.5) { }; 
            \draw (0,0) -- (1.5, 0);
 
            \filldraw (1.5,0) circle (2pt);
            \node at (1.5,-0.5) { };
            \draw (1.5,0) -- (2.5, 0);
            
            \node at (3.5,0) {$\dots$};
            
            \draw (4.5,0) -- (5.5, 0);
            \filldraw (5.5,0) circle (2pt);
            \node at (5.5,-0.5) { };
      
            \draw (5.5,0) -- (7, 0);
            \filldraw (7,0) circle (2pt);
            \node at (7,-0.5) { };
            
            \filldraw (3.5,-3.5) circle (2pt);
            \node at (3.5,-4) {$a$};
            \draw[rotate=0,dashed] (0,0) -- (3.5,-3.5);
            \draw[rotate=0,dashed] (3.5,-3.5) -- (7,0);

         \end{scope}, yshift = 3.5cm

      \end{tikzpicture}
				 
   \end{figure}

  \begin{figure}[ht]
      \begin{tikzpicture}

         \begin{scope}[scale =  0.55, yshift = -9.5cm, xshift = -0cm]
            \filldraw (-1.5,0) circle (2pt);
            \node at (-1.5,-0.5) { }; 
            
            \draw (-1.5,-0.07) -- (0, -0.07);
            \draw (-1.5,0.07) -- (0, 0.07);
            
            \filldraw (0,0) circle (2pt);
            \node at (0,-0.5) { }; 
            
            \draw (0,0) -- (1.5, 0);
            \filldraw (1.5,0) circle (2pt);
            \node at (1.5,-0.5) { };
            
            \draw (1.5,0) -- (3, 0);
            
            \node at (3.5,0) {$\dots$};
            
            \draw (4,0) -- (5.5, 0);
            
            \filldraw (5.5,0) circle (2pt);
            \node at (5.5,-0.5) { };
            
            \draw (5.5,0) -- (7, 0);
            
            \filldraw (7,0) circle (2pt);
            \node at (7,-0.5) { }; 
         \end{scope}
         
         \begin{scope}[scale =  0.55, yshift = -9.5cm, xshift = 15.5cm]
            \draw[line width =1.7pt][->] (-5.5,0) -- (-2.5,0);
            \filldraw (-1.5,0) circle (2pt);
            \node at (-1.5,-0.5) { }; 
            
            \draw (-1.5,-0.07) -- (0, -0.07);
            \draw (-1.5,0.07) -- (0, 0.07);
            
            \filldraw (0,0) circle (2pt);
            \node at (0,-0.5) { }; 
            
            \draw (0,0) -- (1.5, 0);
            \filldraw (1.5,0) circle (2pt);
            \node at (1.5,-0.5) { };
            
            \draw (1.5,0) -- (3, 0);
            
            \node at (3.5,0) {$\dots$};
            
            \draw (4,0) -- (5.5, 0);
            
            \filldraw (5.5,0) circle (2pt);
            \node at (5.5,-0.5) { };
            
            \draw (5.5,0) -- (7, 0);
            
            \filldraw (7,0) circle (2pt);
            \node at (7,-0.5) { }; 
            
            \draw[dashed] (7,-0.07) -- (8.5, -0.07);
            \draw[dashed] (7,0.07) -- (8.5, 0.07);
            
            \filldraw (8.5,0) circle (2pt);
            \node at (8.5,-0.5) {$a$}; 
         \end{scope}

      \end{tikzpicture}
				
   \end{figure}

  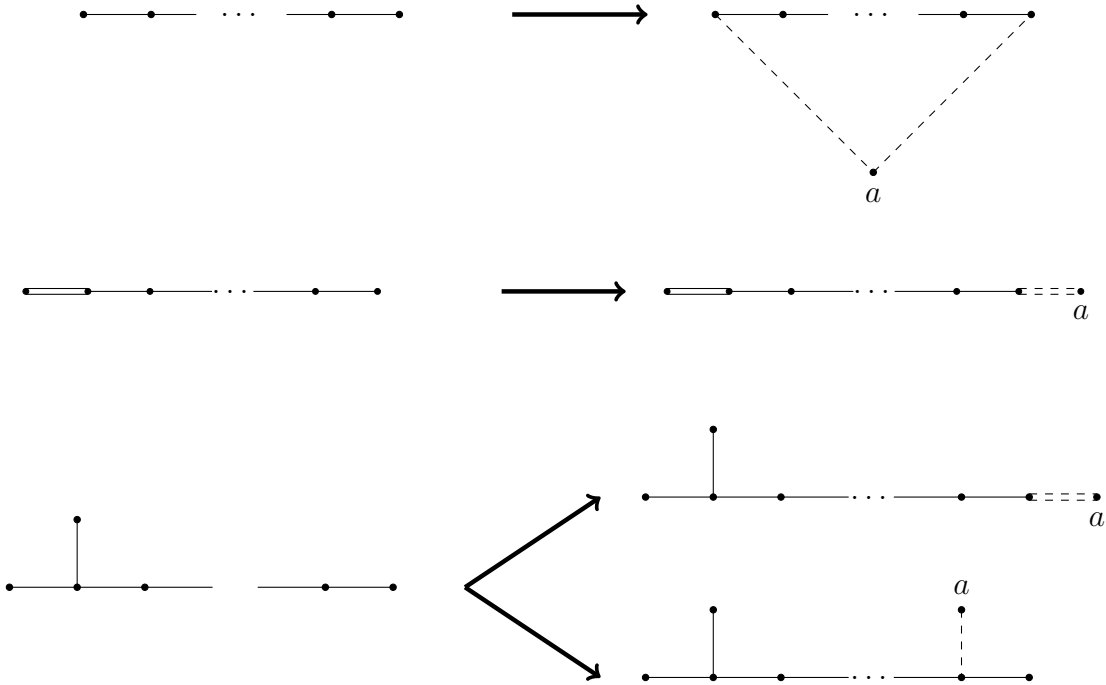
\begin{figure}[ht]
      \begin{tikzpicture}

         \begin{scope}[scale =  0.6, yshift = -18.5cm, xshift =40]
            \filldraw (-1.5,0) circle (2pt);
            \node at (-1.5,-0.5) { }; 
            
            \draw (-1.5,0) -- (0, 0);
            
            \filldraw (0,1.5) circle (2pt);
            \node at (0,2) { }; 
            
            \draw (0,1.5) -- (0, 0);
            
            \filldraw (0,0) circle (2pt);
            
            \node at (0,-0.5) { }; 
            
            \draw (0,0) -- (1.5, 0);
            
            \filldraw (1.5,0) circle (2pt);
            
            \node at (1.5,-0.5) { };
            
            \draw (1.5,0) -- (3, 0);
            
            \node at (3.5,0) { };
            
            \draw (4,0) -- (5.5, 0);
            
            \filldraw (5.5,0) circle (2pt);
            
            \node at (5.5,-0.5) { };
            
            \draw (5.5,0) -- (7, 0);
            
            \filldraw (7,0) circle (2pt);
            \node at (7,-0.5) { };            
         \end{scope}
         
         \begin{scope}[scale =  0.6, yshift = -16.5cm, xshift = 15.5cm]
         
            \draw[line width =1.7pt][->] (-5.5,-2) -- (-2.5,0);
            
            \filldraw (-1.5,0) circle (2pt);
            \node at (-1.5,-0.5) { }; 
            
            \draw (-1.5,0) -- (0, 0);
            
            \filldraw (0,1.5) circle (2pt);
            \node at (0,2) { }; 
            
            \draw (0,1.5) -- (0, 0);
            
            \filldraw (0,0) circle (2pt);
            
            \node at (0,-0.5) { }; 
            
            \draw (0,0) -- (1.5, 0);
            
            \filldraw (1.5,0) circle (2pt);
            
            \node at (1.5,-0.5) { };
            
            \draw (1.5,0) -- (3, 0);
            
            \node at (3.5,0) {$\dots$};
            
            \draw (4,0) -- (5.5, 0);
            
            \filldraw (5.5,0) circle (2pt);
            
            \node at (5.5,-0.5) { };
            
            \draw (5.5,0) -- (7, 0);
            
            \filldraw (7,0) circle (2pt);
            \node at (7,-0.5) { };  
               
            \draw[dashed] (7,-0.07) -- (8.5, -0.07);
            \draw[dashed] (7,0.07) -- (8.5, 0.07);
            
            \filldraw (8.5,0) circle (2pt);
            \node at (8.5,-0.5) {$a$};

         \end{scope}

         \begin{scope}[scale =  0.6, yshift = -20.5cm, xshift = 15.5cm]
         
            \draw[line width =1.7pt][->] (-5.5,2) -- (-2.5,0);
            
            \filldraw (-1.5,0) circle (2pt);
            \node at (-1.5,-0.5) { }; 
            
            \draw (-1.5,0) -- (0, 0);
            
            \filldraw (0,1.5) circle (2pt);
            \node at (0,2) { }; 
            
            \draw (0,1.5) -- (0, 0);
            
            \filldraw (0,0) circle (2pt);
            
            \node at (0,-0.5) { }; 
            
            \draw (0,0) -- (1.5, 0);
            
            \filldraw (1.5,0) circle (2pt);
            
            \node at (1.5,-0.5) { };
            
            \draw (1.5,0) -- (3, 0);
            
            \node at (3.5,0) {$\dots$};
            
            \draw (4,0) -- (5.5, 0);
            
            \filldraw (5.5,0) circle (2pt);
            \filldraw (5.5,1.5) circle (2pt);

            \draw[dashed] (5.5,1.5) -- (5.5, 0);
            \node at (5.5,2) {$a$};

            \node at (5.5,-0.5) { };
            
            \draw (5.5,0) -- (7, 0);
            
            \filldraw (7,0) circle (2pt);
            \node at (7,-0.5) { };            
         \end{scope}

      \end{tikzpicture}
				  \caption{$A$ to  $\tilde A$ ;  $B$ to $\tilde C$ ; $D$ to  $\tilde B$ and $\tilde D$}
   \end{figure}

We thus define the {\it affine length}  of any fully commutative element  to be the number of occurrences of $a$ in a (hence any) reduced expression for this fully commutative element. More general: Let $(W,S)$ be any Coxeter group;  for any $w$ in $W ^c$, let $a$ be any element of $S$, then the number of occurrences of $a$ in any reduced expression of $w$ is independent of the reduced expression, by Matsumoto's theorem.

Moreover, by distinguishing this $a$ we see simply and directly the construction of towers of each type, 
based on the morphisms   in \S\ref{themorphisms}  that map, say, $a_n$ to a suitable conjugate of $a_{n+1}$ while fixing the other generators. In types  $\tilde B$ and $\tilde C$, these morphisms $L_n$ and $F_n$ have a remarkable property: the image of $a_n$ is rigid, as  mentioned previously. Furthermore, for those two types, the braid relation 
involving $a$ has even length so that  the {\it affine length} can be defined for any element of the Coxeter group.
Eventually we prove that these two morphisms preserve reduced expressions (Corollary \ref{longueurBC}), a key argument for proving the injectivity of the tower of Hecke algebras.    \\

This work is to be presented as follows: in the  section 2 we give our notation for the presentation of the Coxeter groups we will be working on, including the "affinization"    explained above. We define morphisms between them  that give rise to injective towers of affine Coxeter groups  of a given type (Corollary \ref{Coxeterinj}). With a closer examination we prove that in types $\tilde B$ and $\tilde C$ those morphisms preserve reduced expressions
(Theorem \ref{reduced} and 
Corollary \ref{longueurBC}).  

Then in section 3   we define the corresponding towers of Hecke algebras. We give a proof of their faithfulness in types  $\tilde B$  and $\tilde C$ over any ring (i.e. commutative ring with identity) where $q$ is invertible. The proof relies on Corollary \ref{longueurBC} that is not available in type $\tilde D$. For type $\tilde D$ the faithfulness of the tower of Hecke algebras can be proven on $\mathbb Q[q, q^{-1}]$ as was done before in \cite{Sadek_Thesis}. 
 \\

In section 4 we recall the normal form given by Stembridge for $D$-type fully commutative elements  \cite[Theorem 10.1]{St} and  present our normal form for $\tilde B$-type and $\tilde D$-type fully commutative elements.  
In both cases we distinguish between two types of elements of affine length at least 2, that we call simply {\it first type} elements and   {\it second type} elements.   Elements of affine length 1 cannot be attached to these two categories 
without ambiguity so we keep them apart.  

We show in   section 5  that  the normal form of a  fully commutative element in 
$W(\tilde B_n)$ (resp.$W(\tilde D_n)$   can be transformed  into  the normal form of a  fully commutative element in 
$W(\tilde B_{n+1})$ (resp.$W(\tilde D_{n+1})$  in two different ways, that actually coincide for first type elements and elements of affine length one. We obtain  two injective maps $I$ and $J$ that play an essential part 
  in section 6, where we study the morphisms of Temperley-Lieb algebras  
$Q_{n}: TL\tilde{B}_{n}(q)  \longrightarrow  TL\tilde{B}_{n+1}(q) $ and  
$P_{n}: TL\tilde{D}_{n}(q) \longrightarrow  TL\tilde{D}_{n+1}(q)$ induced from the morphisms of Hecke algebras of section 3. 

  Indeed, we can describe the image of a basis element indexed by a fully commutative element $w$ 
as a linear combination of basis elements in which the terms of highest affine length and highest Coxeter length are indexed by 
$I(w)$ and $J(w)$  (Lemma  \ref{formulaB} and Proposition \ref{formula2B}; 
Proposition \ref{formulaLge2D}) or by  $I(w)$ and some $I(\bar w)$ when $w$ has affine length 1 in type $\tilde D$  (Lemma \ref{formulaD}). This implies the faithfulness of the towers of Temperley-Lieb algebras:        Theorem \ref{RB}          and  Theorem \ref{RD}.

\section{Faithful towers  of Coxeter groups}\label{group tower}

Let $(W(\Gamma),S)$ be a Coxeter system with associated Coxeter diagram $\Gamma$. Let $w\in W(\Gamma)$ or simply $W$. We denote by $l(w)$ the length of a (any) reduced expression of $w$.   We define $\mathscr{L} (w) $ to be the set of $s\in S$ such that $l(sw)<l(w)$, in other terms  $s$ appears at the left edge of some reduced expression of $w$.  We define $\mathscr{R}(w)$ similarly, on the right.\\
 
\subsection{Presentations of   $D_{n+1}$,  $\tilde D_{n+1}$, $\tilde B_{n+1}$, $\tilde C_{n+1}$} 

For $ n \ge 3$ consider the $D$-type Coxeter group with $n+1$ generators $W(D_{n+1})$, with the following Coxeter diagram:\\
			
			\begin{figure}[ht]
				\centering
				\begin{tikzpicture}

 \filldraw (-1.5,0) circle (2pt);
  \node at (-1.5,-0.5) {$\sigma_1$}; 

  \draw (-1.5,0) -- (0, 0);

\filldraw (0,1.5) circle (2pt);
 \node at (0,2) {$ \sigma_{\bar 1}$}; 

  \draw (0,1.5) -- (0, 0);
    
  \filldraw (0,0) circle (2pt);
  \node at (0,-0.5) {$\sigma_{2}$}; 
   
  \draw (0,0) -- (1.5, 0);

  \filldraw (1.5,0) circle (2pt);
  \node at (1.5,-0.5) {$\sigma_{3}$};

  \draw (1.5,0) -- (3, 0);

  \node at (3.5,0) {$\dots$};

  \draw (4,0) -- (5.5, 0);
  
  \filldraw (5.5,0) circle (2pt);
  \node at (5.5,-0.5) {$\sigma_{n-1}$};
 
  \draw (5.5,0) -- (7, 0);
  
  \filldraw (7,0) circle (2pt);
  \node at (7,-0.5) {$\sigma_{n}$};

  

               \end{tikzpicture}
				 \caption{$D_{n+1}$}
			\end{figure}

Now let $W(\tilde{D}_{n+1}) $ be the affine Coxeter group of $\tilde{D}$-type with $n+2$ generators in which $W(D_{n+1})$ could be seen a parabolic subgroup in two ways. We make our choice by presenting $W(\tilde{D}_{n+1} ) $ with the following Coxeter diagram: \\
			
			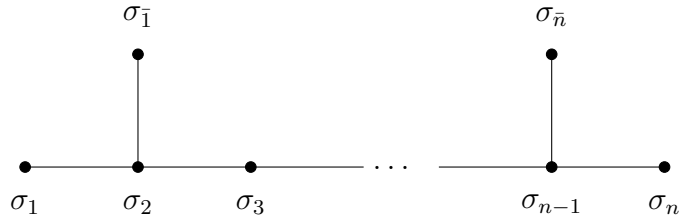
\begin{figure}[ht]
				\centering
				\begin{tikzpicture}

\filldraw (-1.5,0) circle (2pt);
  \node at (-1.5,-0.5) {$\sigma_1$}; 

  \draw (-1.5,0) -- (0, 0);

\filldraw (0,1.5) circle (2pt);
 \node at (0,2) {$ \sigma_{\bar 1}$}; 

  \draw (0,1.5) -- (0, 0);
    
  \filldraw (0,0) circle (2pt);
  \node at (0,-0.5) {$\sigma_{2}$}; 
   
  \draw (0,0) -- (1.5, 0);

  \filldraw (1.5,0) circle (2pt);
  \node at (1.5,-0.5) {$\sigma_{3}$};

  \draw (1.5,0) -- (3, 0);

  \node at (3.4,0) {$\dots$};

  \draw (4,0) -- (5.5, 0);
  
  \filldraw (5.5,0) circle (2pt);
  \node at (5.5,-0.5) {$\sigma_{n-1}$};
 
  \draw (5.5,0) -- (7, 0);
  
  \filldraw (5.5,1.5) circle (2pt);
 \node at (5.5,2) {$ \sigma_{\bar {n}}$}; 

  \draw (5.5,1.5) -- (5.5,0);
  
  \filldraw (7,0) circle (2pt);
  \node at (7,-0.5) {$\sigma_{n}$};

               \end{tikzpicture}
			 \caption{$\tilde D_{n+1}$}
			\end{figure}
In other words the group  $W(\tilde{D}_{n+1} ) $ has a presentation given by the set of generators 
$\{ \sigma_{1},  \sigma_{\bar 1}  \dots,  \sigma_{n-1},\sigma_{\bar {n}}, \sigma_{n} \}$ and the relations: 
$$ \begin{aligned}
&\sigma_{\bar 1}^2 =  \sigma_{\bar {n}}^2 =1 \text{ and } \sigma_i^2 = 1  \text{ for } 1\le i \le n ; \\ 
&\sigma_i \sigma_{j} = \sigma_{j} \sigma_i \text{ for } 1\le i, j \le n, \ |i-j|\ge  2 ; \\
&\sigma_i \sigma_{\bar 1} = \sigma_{\bar 1} \sigma_i ~~~~~~ \text{ for } i \not= 2 ; \\
&\sigma_i \sigma_{\bar {n}} =\sigma_{\bar {n}} \sigma_i \text{ for } i\not=n-1 ; \\
&\sigma_i \sigma_{i+1} \sigma_i = \sigma_{i+1} \sigma_i\sigma_{i+1}  \text{ for } 1\le i \le n-1;\\ 
&\sigma_2 \sigma_{\bar 1} \sigma_2 = \sigma_{\bar 1} \sigma_2\sigma_{\bar 1} ;\\ 
&\sigma_{n-1} \sigma_{\bar {n}} \sigma_{n-1} = \sigma_{\bar {n}} \sigma_{n-1}\sigma_{\bar {n}};\\ 
\end{aligned}
$$

Now let $W(\tilde{B}_{n+1}) $ be the affine Coxeter group of $\tilde{B}$-type with $n+2$ generators in which $W(D_{n+1})$ is naturally a parabolic subgroup, as seen in  the following Coxeter diagram: \\
			
				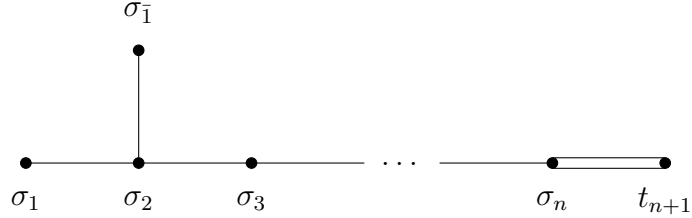
\begin{figure}[ht]
				\centering
				\begin{tikzpicture}

 \filldraw (-1.5,0) circle (2pt);
  \node at (-1.5,-0.5) {$\sigma_1$}; 

  \draw (-1.5,0) -- (0, 0);

\filldraw (0,1.5) circle (2pt);
 \node at (0,2) {$ \sigma_{\bar 1}$}; 

  \draw (0,1.5) -- (0, 0);
  
  \filldraw (0,0) circle (2pt);
  \node at (0,-0.5) {$\sigma_{2}$}; 
   
  \draw (0,0) -- (1.5, 0);

  \filldraw (1.5,0) circle (2pt);
  \node at (1.5,-0.5) {$\sigma_{3}$};

  \draw (1.5,0) -- (3, 0);

  \node at (3.5,0) {$\dots$};

  \draw (4,0) -- (5.5, 0);
  
  \filldraw (5.5,0) circle (2pt);
  \node at (5.5,-0.5) {$\sigma_{n}$};
 
  \draw (5.5,-0.07) -- (7, -0.07);
\draw (5.5,0.07) -- (7, 0.07);
  
  \filldraw (7,0) circle (2pt);
  \node at (7,-0.5) {$t_{n+1}$};

               \end{tikzpicture}
			 \caption{$\tilde B_{n+1}$}
			\end{figure}
In other words the group  $W(\tilde{B}_{n+1} ) $ has a presentation given by the set of generators 
$\{ \sigma_{1},  \sigma_{\bar 1}, \dots, \sigma_{n},t_{n+1} \}$ and the relations: 
$$ \begin{aligned}
&  t_{n+1}^2 = 1,   {\sigma_{\bar 1}}^2=1  \text{ and } \sigma_i^2 = 1  \text{ for } 1\le i \le n ; \\ 
&\sigma_i \sigma_{j} = \sigma_{j} \sigma_i \text{ for } 1\le i, j \le n, \ |i-j|\ge  2 ; \\
&\sigma_i t_{n+1} = t_{n+1}\sigma_i \text{ for } 1\le i < n ; \\ 
& \sigma_{\bar 1}  t_{n+1} = t_{n+1} \sigma_{\bar 1}    ; \\ 
&\sigma_i  \sigma_{\bar 1} = \sigma_{\bar 1}\sigma_i \text{ for } i=1 \text{ or } 3\le i   ; \\ 
&\sigma_i \sigma_{i+1} \sigma_i = \sigma_{i+1} \sigma_i\sigma_{i+1}  \text{ for } 1\le i \le n-1;\\ 
& \sigma_{\bar 1} \sigma_{2} \sigma_{\bar 1} =  \sigma_{2} \sigma_{\bar 1}\sigma_{2}   ; \\
&\sigma_n t_{n+1}\sigma_n t_{n+1}= t_{n+1}\sigma_n t_{n+1}\sigma_n. \\  
\end{aligned}
$$

\medskip		

Now we consider the  $\tilde C$-type Coxeter group with $n+2$ generators $W(\tilde{C}_{n+1}) $, having as Coxeter diagram:

			\begin{figure}[ht]
				\centering
				\begin{tikzpicture}

 \filldraw (-1.5,0) circle (2pt);
  \node at (-1.5,-0.5) {$t=\sigma_0$}; 

  \draw (-1.5,-0.07) -- (0, -0.07);
  \draw (-1.5,0.07) -- (0, 0.07);
  
  \filldraw (0,0) circle (2pt);
  \node at (0,-0.5) {$\sigma_{1}$}; 
   
  \draw (0,0) -- (1.5, 0);

  \filldraw (1.5,0) circle (2pt);
  \node at (1.5,-0.5) {$\sigma_{2}$};

  \draw (1.5,0) -- (3, 0);

  \node at (3.5,0) {$\dots$};

  \draw (4,0) -- (5.5, 0);
  
  \filldraw (5.5,0) circle (2pt);
  \node at (5.5,-0.5) {$\sigma_{n-1}$};
 
  \draw (5.5,0) -- (7, 0);
  
  \filldraw (7,0) circle (2pt);
  \node at (7,-0.5) {$\sigma_{n}$};

   \draw (7,-0.07) -- (8.5, -0.07);
  \draw (7,0.07) -- (8.5, 0.07);
  
   \filldraw (8.5,0) circle (2pt);
  \node at (8.5,-0.5) {$t_{n+1}$};

               \end{tikzpicture}
				 \caption{$\tilde C_{n+1}$}
			\end{figure}
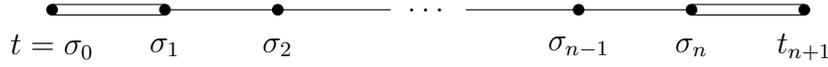

In other terms we generate $W(\tilde{C}_{n+1}) $ by 	 $\{ \sigma_{0}, \sigma_{1}, \dots, \sigma_{n},t_{n+1}\}$  subject to the following relations: 

$$ \begin{aligned}
&  t_{n+1}^2 = 1 \text{ and } \sigma_i^2 = 1  \text{ for } 0\le i \le n ; \\ 
&\sigma_i \sigma_{j} = \sigma_{j} \sigma_i \text{ for } 0\le i, j \le n, \ |i-j|\ge  2 ; \\
&\sigma_i t_{n+1} = t_{n+1}\sigma_i \text{ for } 0\le i < n ; \\ 
&\sigma_i \sigma_{i+1} \sigma_i = \sigma_{i+1} \sigma_i\sigma_{i+1}  \text{ for } 1\le i \le n-1;\\ 
&\sigma_0 \sigma_{1}\sigma_0 \sigma_{1} =  \sigma_{1}\sigma_0 \sigma_{1}\sigma_0  ; \\
&\sigma_n t_{n+1}\sigma_n t_{n+1}= t_{n+1}\sigma_n t_{n+1}\sigma_n.\\ \\
\end{aligned}
$$

\subsection{Faithfulness of towers of $\tilde B$-type and $\tilde D$-type affine Coxeter groups}\label{themorphisms}
		
In \cite[Corollary 2.2]{Sadek_2017} the author has defined a faithful tower of   $\tilde C$-type Coxeter groups  by defining  the following homomorphism for $n \ge 2$ and proving its
 injectivity: 
\begin{eqnarray}
					  F_n: W(\tilde C_{n }) &\longrightarrow& W(\tilde C_{n+1})\nonumber\\
					\sigma_{i} &\longmapsto& \sigma_{i}$ ~~~ \text{for} $0\leq i\leq n-1 \nonumber\\
					t_{n} &\longmapsto& {\sigma_n} t_{n+1} {\sigma_n} 
	\\ \nonumber
				\end{eqnarray}
We define in a similar way the two following homomorphisms: 
\begin{eqnarray}\label{Ln}   
					  L_n: W(\tilde B_{n }) &\longrightarrow& W(\tilde B_{n+1})\nonumber\\
					\sigma_{i} &\longmapsto& \sigma_{i}$ ~~~ \text{for} $1\leq i\leq n-1 \nonumber\\
					\sigma_{\bar 1} &\longmapsto& \sigma_{\bar 1}  \nonumber\\
					t_{n} &\longmapsto& {\sigma_n} t_{n+1} {\sigma_n} 
	\\ \nonumber
				\end{eqnarray}
\begin{eqnarray}
					  G_n: W(\tilde D_{n }) &\longrightarrow& W(\tilde D_{n+1})\nonumber\\
					\sigma_{i} &\longmapsto& \sigma_{i}$ ~~~ \text{for} $1\leq i\leq n-1 \nonumber\\
					\sigma_{\bar 1} &\longmapsto& \sigma_{\bar 1}  \nonumber\\
					\sigma_{\bar{n-1}} &\longmapsto&\sigma_n \sigma_{n-1} \sigma_{\bar {n}} \sigma_{n-1}\sigma_n  \\
 \nonumber
				\end{eqnarray}

In 
\cite[\S 1.13]{Lusztig} we see that for $n>2$, in $W(\tilde{C}_{n+1} ) $ the subgroup generated by  $\{ t\sigma_{1}t, \sigma_{1}, \sigma_{2}, \dots,  \sigma_{n}, t_{n+1} \}$ is exactly $W(\tilde{B}_{n+1} ) $, and in the later the subgroup generated by $\{ t\sigma_{1}t, \sigma_{1}, \sigma_{2}, \dots, \sigma_{n}, t_{n+1}\sigma_{n}t_{n+1}  \}$ is indeed $W(\tilde{D}_{n+1} ) $. It is convenient to denote respectively $\beta $ and $\delta_{n+1} $ those inclusion maps (that is, 
$\beta (\sigma_{\bar 1})=t\sigma_{1}t$ and  $\delta_{n+1}( \sigma_{\bar {n}})=t_{n+1}\sigma_{n}t_{n+1} $   , in such a way that we have the following lemma: \\

 \begin{lemma} 
     
     The diagram commutes  for any $n > 2$

\begin{center}    

	\begin{tikzpicture}

			\matrix[matrix of math nodes,row sep=1cm,column sep=1cm]{
			|(A)| W(\tilde{C}_{n})    & & & &    |(B)| W(\tilde{C}_{n+1} ) \\
			                              & & & &                      \\								
			|(C)|   W(\tilde{B}_{n} )             & & & &    |(D)|   W(\tilde{B}_{n+1} )      \\
			                             & & & &                      \\								
			|(E)|   W(\tilde{D}_{n} )             & & & &    |(F)|   W(\tilde{D}_{n+1} )      \\
				};

				\path (A) edge[-myhook,line width=0.42pt]  node[above, xshift=-5mm, yshift=-2mm, rotate=0] 
{\footnotesize $\beta $}    (C);
\path (C) edge[-myto,line width=0.42pt]      (A);

	\path (C) edge[-myhook,line width=0.42pt]  node[above, xshift=-5mm, yshift=-2mm, rotate=0] 
{\footnotesize $\delta_{n} $}    (E);
\path (E) edge[-myto,line width=0.42pt]      (C);
				
					\path (B) edge[-myhook,line width=0.42pt]  node[above, xshift=1.5mm, yshift=0mm, rotate=0] {\footnotesize $F_{n}$}      (A);
\path (A)  edge[-myto,line width=0.42pt]    (B);

				\path (C) edge[-myto,line width=0.42pt]   node[below, xshift=1.5mm, yshift=0mm, rotate=0]    {\footnotesize $L_{n}$}      (D);

				\path (E) edge[-myto,line width=0.42pt]   node[below, xshift=1.5mm, yshift=0mm, rotate=0]    {\footnotesize $G_{n}$}      (F);

				 \path (B) edge[-myhook,line width=0.42pt]  node[above, xshift=-5mm, yshift=-2mm, rotate=0] 
{\footnotesize $\beta $} (D);
				\path (D) edge[-myto,line width=0.42pt]     (B);

				 \path (D) edge[-myhook,line width=0.42pt]  node[above, xshift=-5mm, yshift=-2mm, rotate=0] 
{\footnotesize $\delta_{n+1} $}    (F);
				\path (F) edge[-myto,line width=0.42pt]     (D);

		\end{tikzpicture}	
	\end{center}

     \end{lemma}     
     
       \begin{corollary}\label{Coxeterinj}  $L_n$ and $G_n$ are injections. \\  \end{corollary}

\begin{remark}\label{braid}  {\rm
The horizontal morphisms in the previous lemma, namely $F_n$, $L_n$   and $G_n$,  can be defined on the corresponding braid groups (by mapping $t_n$ to 
${\sigma_n} t_{n+1} {\sigma_n}^{-1}$ for $F_n$ and $L_n$, or mapping 
	$\sigma_{\bar{n-1}}$ to $\sigma_n \sigma_{n-1} \sigma_{\bar {n}} \sigma_{n-1}^{-1}\sigma_n^{-1}$
for $G_n$). They thus define morphisms of algebra on the braid group algebras and, after taking quotients, on the Hecke algebras studied in the next section. This is not the case for   $\delta_n$ or $\beta$.}\\  
\end{remark}

\subsection{Further properties of the towers in type $\tilde B$ and $\tilde C$}  

We now remark that,  in the groups $W(\tilde{B}_{n+1}) $ and $W(\tilde{C}_{n+1}) $, the only braid relation involving $t_{n+1}$ (apart from commutation relations) is 
\begin{eqnarray}\label{braid4}
t_{n+1}\sigma_n t_{n+1}\sigma_n = \sigma_n t_{n+1}\sigma_n t_{n+1}
\end{eqnarray}
 where  the number of occurrences of 
$t_{n+1}$ is the same on both sides. It follows (recall \cite[\S 1.5 Proposition 5]{Bourbaki_1981}) that  the number of times   $t_{n+1}$ occurs in a 
 reduced expression of an element of   $ W(\tilde{B}_{n+1}) $  or $W(\tilde{C}_{n+1}) $ does not depend of this reduced expression.

			\begin{definition} \label{ALB} Let $u \in W (\tilde B_{n+1})$ or $W(\tilde{C}_{n+1}) $. 
				We define the {\rm affine length}  of $u$  to be the number of times  $t_{n+1}$ occurs   in a (any) 
 reduced expression of $u$. We denote  it by $L(u)$. 
			\end{definition}

In the group $W(\tilde{D}_{n+1}) $, the braid relation 
 $\sigma_{n-1}\sigma_{\bar n}\sigma_{n-1} = \sigma_{\bar n} \sigma_{n-1}\sigma_{\bar n} $ 
prevents us from giving the same definition, except for fully commutative 	elements as we will see below 
(Definition \ref{ALD}).\\

We will now examine more closely the monomorphisms  $F_n: W(\tilde C_{n })  \longrightarrow  W(\tilde C_{n+1})$ 
and   $L_n: W(\tilde B_{n })  \longrightarrow  W(\tilde B_{n+1})$ 
 and show that they preserve the affine length and transform the Coxeter length of some $w$ into  
$l(w)+2L(w)$. In other words, when substituting ${\sigma_n} t_{n+1} {\sigma_n} $ to $t_n$ in a reduced expression for $w$, we obtain a reduced expression for the image of $w$. 

We let in this subsection $W_n$ be $W(\tilde C_{n })$ or  $W(\tilde B_{n })$, with set of generators 
$S_n = \Sigma \cup \{t_n\}$  while $W_{n+1}$ has $S_{n+1} = \Sigma \cup \{\sigma_n, t_{n+1}\}$ as set of generators. 
We let $I$ be either $F_n$ or $L_n$, i.e. $I$ is the monomorphism from $W_n$ to $W_{n+1}$ that is the identity on generators in $\Sigma$ and maps $t_n$ to $\sigma_n  t_{n+1} \sigma_n$. We also let $W_\Sigma$ be the subgroup of $W_n$ generated by $\Sigma$.

\begin{lemma}

Let $x$ be in $I(W_n)$ then:  \\

(a) $t_{n+1} x= xt_{n+1}$ ; \\

(b) If $\sigma_n x= x\sigma_n$ then  $I(t_{n}) x= x I(t_{n})$; \\

(c) If $xt_{n+1} \sigma_n  = \sigma_n t_{n+1}x$ then $I(t_{n}) x= x I(t_{n})$; \\

  \end{lemma}

  \begin{proof} (a) Indeed $ t_{n+1}$ commutes with generators in $\Sigma$ and the braid relation (\ref{braid4}) expresses commutation with $T(t_n)$. 
(b)  follows immediately. 

Using (a) the assumption in (c) can be written   $x \sigma_n  x^{-1}=t_{n+1} \sigma_n t_{n+1}$, hence by  (\ref{braid4}): 
   $$x \sigma_n  x^{-1}=  \sigma_n t_{n+1} \sigma_n t_{n+1} \sigma_n=  (\sigma_n t_{n+1} \sigma_n) \sigma_n (\sigma_n t_{n+1} \sigma_n )=I( t_n )\sigma_n I( t_n).$$ 
So $I( t_n )x$ commutes with $\sigma_n  $, hence with $I( t_n )$ by (b), which gives   

$I( t_n )I( t_n ) x =     I( t_n )  x  I( t_n )  $, that is $   I( t_n )  x   =x  I( t_n )$. 
  \end{proof}

We now remember \cite[Ch. IV, \S 1.4]{Bourbaki_1981}. Given a Coxeter system $(W,S)$, we attach to any finite sequence 
$\mathbf s = (s_1, \cdots, s_r)$ of elements in $S$,
 the multiset $\Phi(\mathbf s)= \{\{h_1, \cdots, h_r\}\}$ of elements in $W$ defined by: 
$$ 
h_j = (s_1\cdots s_{j-1}) \ s_j \   (s_1\cdots s_{j-1})^{-1}   \qquad \text{ for } 1 \le j \le r .   
$$
We write $\MCard \Phi(\mathbf s) $ for the multi-cardinal of  $\Phi(\mathbf s)$, i.e.the number of elements in the multiset $\Phi(\mathbf s)$, namely $r$ here, and 
$\Card \Phi(\mathbf s)$ for the cardinality of the set underlying $\Phi(\mathbf s)$. We know  from 
 {\it loc.cit.} Lemma 2  that:  

{\it  the product  $s_1 \cdots s_r$ is a reduced expression of the element  $s_1 \cdots s_r$ in $W$ if and only if  $  \  \MCard \Phi(\mathbf s)  =\Card \Phi(\mathbf s)$, that is, all elements in the multiset are distinct}. \\

We now let $w \in W_n$ and fix a reduced decomposition of $w$ as follows : 
$$  w =  u_0 t_{n} u_1 \dots u_{s-1} t_{n}u_s
$$
where $u_0, \cdots, u_s$ are fixed reduced decompositions of elements in $W_\Sigma$. In particular the affine length of $w$ is $s$.  
The image $I(w)$  in $W_{n+1}$ is $$
(\ast) \qquad    I(w)= u_0 \sigma_n t_{n+1} \sigma_n  u_1 \dots u_{s-1} \sigma_n t_{n+1} \sigma_n u_s.$$ 
 We use this fixed expression of $I(w)$   to study what we call, by abuse of language, $\Phi(I(w))$. 
We can describe this multiset as a disjoint union of multisets 
 $$\Phi(I(w))=I(\Phi( w))  \bigsqcup  T_1  \bigsqcup  T_2$$
as follows. 
We develop the expressions of $u_0 , \dots, u_s$ so that $(\ast)$ reads 
$I(w)= s_1 \cdots s_r$, then we have three possibilities for an element 
$ 
h_i = (s_1\cdots s_{i-1}) \ s_i \   (s_1\cdots s_{i-1})^{-1}   
$ of   $\Phi(I(w))$:

 \begin{enumerate}
\item $s_i \in \Sigma$ or $s_i= t_{n+1} $; then 
  $h_i= x X x^{-1}$ where $x$ and $X$ are in  $I(W_n)$. Those elements  make up exactly the multiset 
 $I(\Phi( w))$. 
\item $s_i= \sigma_n$ and $s_{i+1}=t_{n+1}$; then    $h_i= x\sigma_nx^{-1}$ with  $x$   in  $I(W_n)$.  
Those elements make up the multiset   $T_1$.  
\item $s_i= \sigma_n$ and $s_{i-1}=t_{n+1}$; then  $h_i= x \sigma_n  t_{n+1} \sigma_n  t_{n+1}  \sigma_n  x^{-1} =  x  t_{n+1} \sigma_n t_{n+1} x^{-1}$  with  $x$   in  $I(W_n)$. 
Those elements make up the multiset   $T_2$.  \\  
  \end{enumerate}

  We have $\MCard (T_1) =  \MCard (T_2) = L(w)$  and  $\MCard(I(\Phi( w)))= l(w) $,  since $w$ is given in a reduced expression (we will anyway  see it in (1) of the following theorem). 
 Hence the multi-cardinal  of $\Phi(I(w))$ is $l(w)+2L(w)$. If the elements of $\Phi(I(w))$ are all distinct, then it is the cardinal.\\
  
  \begin{theorem}\label{reduced}
Let $w \in W_n$ and take the above conventions. Then   $$
MCard(\Phi(I(w))) = Card(\Phi(I(w))) = l(w)+2L(w).
$$   
In particular,     substituting $ \  {\sigma_n} t_{n+1} {\sigma_n} \  $ to $ \  t_n \  $ in a reduced expression for $w$ produces a reduced expression for the image $I(w)$ of $w$. \\ 
  
  \end{theorem}
  
  \begin{proof}
  
  Take $h_i$ and $h_j$ in $ \Phi(I(w))$ with $i\ne j$ and suppose that $h_i =h_j$. We have six cases.   
  
  (1) $h_i$ and $h_j$ are in $I(\Phi( w))$: contradicts that $w$ is given in its reduced expression  and $I$ is injective.
  
  (2)  $h_i$ and $h_j$ are in $T_1 $, gives (b) in the above lemma, thus contradicts that $w$ is given in its reduced expression.
  
  (3)  $h_i$ and $h_j$ are in $T_2 $,  gives (b) in the above lemma, thus contradicts that $w$ is given in its reduced expression.
  
  (4) $h_i$ is in $I(\Phi( w))$ and  $h_j$ is in $T_1 $, gives $\sigma_n  \in  I(W_n)$, hence commutes with  $t_{n+1}$ by  (a) in the above lemma,  contradiction.
  
  (5) $h_i$ is in $I(\Phi( w))$ and  $h_j$ is in $T_2 $, gives $t_{n+1} \sigma_n t_{n+1}  \in I(W_n)$, hence commutes with  $t_{n+1}$ by  (a) in the above lemma, contradiction.
  
  (6) $h_i$ is in $T_1 $ and  $h_j$ is in $T_2 $,  gives (c) in the above lemma, thus contradicts that $w$ is given in its reduced expression.
  \end{proof}

\begin{corollary}\label{longueurBC}
The monomorphisms  
$$ \begin{aligned}
F_n: W(\tilde C_{n })  &\longrightarrow  W(\tilde C_{n+1}) 
    \\  \text{ and }  \  \   
					 L_n: W(\tilde B_{n })  &\longrightarrow  W(\tilde B_{n+1})
\end{aligned}
 $$ satisfy, for any $v \in W(\tilde{C}_{n} )$ and for any $w \in W(\tilde{B}_{n} )$ : 
$$ \begin{aligned}
l(F_n(v))& = l(v) + 2 L(v) \ \   \text{ and } L(F_n(v)) =L(v) ; \\   l(L_n(w))&= l(w)+2L(w)  \  \text{ and } L(L_n(w)) =L(w) .   
\end{aligned}$$

\end{corollary}

\bigskip

\section{The towers of  Hecke algebras}\label{HC}      

Let for the moment $K$ be an arbitrary commutative ring with identity; 
 we mean by algebra in what follows $K$-algebra. We recall \cite[Ch. IV \S 2 Ex. 23]{Bourbaki_1981} that for  a given Coxeter graph $\Gamma$ and a corresponding Coxeter system $(W(\Gamma),S)$,   there is a unique algebra structure on the free $K$-module with basis $ \left\{  g_w \ | \ w \in W(\Gamma) \right\} $ satisfying, for $s \in S$ and a given $q \in K$: 
\begin{equation*}\label{definingrelations} 	
	 \begin{aligned}
		 &g_{s} g_{w} =g_{sw}     ~~~~~~~~~~~~~~~~~~~~  \text{ if } s \notin \mathscr{L} (w) , \\
		 &g_{s} g_{w} =qg_{sw}+ (q-1)g_w ~~  \text{ if } s \in \mathscr{L} (w). \\
			  \end{aligned}    \qquad 
		\end{equation*}
We denote this algebra by $H\Gamma(q)$  and call it the the $\Gamma$-type Hecke algebra.   This algebra has a presentation ({\it loc.cit.}) given by generators $ \left\{  g_s \ | \   s \in S \right\} $ and relations 
$$
	 \begin{aligned}
		  g_{s}^2   &=q + (q-1)g_s ~~~~~~~~ \   \text {for } s \in S,  \\
		  (g_{s} g_{t})^r &=(g_{t} g_{s})^r    ~~~~~~~~~~~~~~~~~   \! \text{ for } s, t \in S 
\text{ such that } st \text{ has order }  2r , \\
  (g_{s} g_{t})^r g_s&=(g_{t} g_{s})^r  g_t   ~~~~~~~~~~~~~~  \text{ for } s, t \in S 
\text{ such that } st \text{ has order }  2r+1 .  
			  \end{aligned}   
$$
We   assume in what follows that $q$ is invertible in $K$. In this case  the first defining relation above implies that $ g_{s}$, for $s \in S$, is invertible 
with inverse 
\begin{equation}\label{inverse} 
	g^{-1}_s = 		 \frac{1}{q} \;  	g_s +  \frac{q-1}{q}.
 \end{equation} 

\medskip 

We consider the Hecke algebras  $H\tilde{ C}_{n} (q) $, of  type $\tilde{C_n}$, 
$H\tilde{ B}_{n} (q) $, of  type $\tilde{B_n}$, and $H\tilde{ D}_{n} (q) $, of  type $\tilde{D_n}$, 
which correspond respectively to the Coxeter groups $W(\tilde{C}_{n} )$, $W(\tilde{B}_{n} )$ and $W(\tilde{D}_{n} )$.
		The morphisms $F_n$, $L_n$ and $G_n$ of \S \ref{group tower} have a counterpart in the setting of Hecke algebras, as follows from Remark \ref{braid} or as can be   checked directly,  namely the following morphisms of algebras (where we write carefully $e_w$ for the basis elements of the algebras in rank $n$, to be reminded of the possible lack of injectivity):  
	
		\begin{equation}\label{defRn}
\begin{aligned}
					  R_n: H\tilde{C}_n (q)  &\longrightarrow   H\tilde{C}_{n+1} (q) \\
					e_{\sigma_i} &\longmapsto  g_{\sigma_i}  ~~~ ~~~ ~~~ \text{for }  0\leq i\leq n-1  \\
					e_{t_n} &\longmapsto  g_{\sigma_n} g_{t_{n+1}}g_{\sigma_n}^{-1}  .
\end{aligned}		
\end{equation}

		\begin{equation}\label{defQn}
\begin{aligned}
					  Q_n: H\tilde{B}_n (q)  &\longrightarrow   H\tilde{B}_{n+1} (q) \\
					e_{\sigma } &\longmapsto  g_{\sigma }  ~~~ ~~~ ~~~ \text{for }  
\sigma \in \{ \sigma_{1},  \sigma_{\bar 1}, \dots, \sigma_{n}  \}  \\
					e_{t_n} &\longmapsto  g_{\sigma_n} g_{t_{n+1}}g_{\sigma_n}^{-1}  .
\end{aligned}		
\end{equation}

		\begin{equation}\label{defPn}
\begin{aligned}
					  P_n: H\tilde{D}_n (q)  &\longrightarrow   H\tilde{D}_{n+1} (q) \\
	e_{\sigma } &\longmapsto  g_{\sigma }  ~~~ ~~~ ~~~ \text{for }  
\sigma \in \{ \sigma_{1},  \sigma_{\bar 1}, \dots, \sigma_{n-1}  \}   \\
					e_{\sigma_{\bar {n-1}}} &\longmapsto  g_{\sigma_n}g_{\sigma_{n-1}} g_{\sigma_{\bar n}}g_{\sigma_{n-1}}^{-1} g_{\sigma_n}^{-1}  .
\end{aligned}		
\end{equation}	 

\medskip 

When  $K $  is the ring   $\mathbf Q[q,q^{-1}]$ of Laurent polynomials in $q$ with rational coefficients, we proved in 
\cite[Proposition 3.3]{Sadek_2017} that the morphism $R_n$ is injective,  hence producing a faithful tower of 
Heche algebras of type $\tilde C$. Actually for types $\tilde C$ and $\tilde B$ we prove below injectivity for any ground ring $K$. For type $\tilde D$ we keep to the 
special case $K = \mathbf Q[q,q^{-1}]$ but we will skip the details of the proof, similar to the proof in {\it loc.cit.}.
In any case,  injectivity at the level of Hecke algebras is not needed to prove injectivity for towers of Temperley-Lieb algebras in \S \ref{TL}.

\subsection{Faithfulness of towers of $\tilde B$-type and $\tilde C$-type  Hecke algebras}

For types $\tilde B$ and $\tilde C$ the key is the following Proposition:

\begin{proposition}\label{coroBC}
	Let $w$ be any element in $W(\tilde{C}_{n } )$. Then:
				\begin{eqnarray}
					R_n(e_{w}) =A_w g_{F_n(w)}+ \sum\limits_{\begin{smallmatrix} x\in W(\tilde{C}_{n+1}),    \cr  l(x)<l(F_n(w)) \cr L(x)\le L(w) \end{smallmatrix}} \lambda_{x}g_{x}, \nonumber
				\end{eqnarray}
				where $A_w $ belongs to $ q^\mathds Z$ and  the $ \lambda_{x}$ belong to $K$. \\ 

The same holds for $w \in W(\tilde{B}_{n } )$, replacing $R_n$ by $Q_n$ and $F_n$ by $L_n$. \\ 
\end{proposition} 

\begin{proof} This is a direct consequence of 
  Corollary \ref{longueurBC}. To obtain it we use 
     definitions (\ref{defRn}) of $R_n$ and  (\ref{defQn}) of $Q_n$ and develop 
the images $R_n(e_{w})$ and $Q_n(e_{w}) $ with  the expression (\ref{inverse}) for the inverse of $g_{\sigma_n}$.  
\end{proof}

\begin{theorem}\label{HeckeBC} Let $K$ be a ring and $q$ be invertible in $K$. 
The towers of affine Hecke  algebras: 
\begin{eqnarray}
			  H\tilde{C}_{1}(q)  \stackrel{R_{1}}{\longrightarrow}  H\tilde{C}_{2}(q) \stackrel{R_{2}} {\longrightarrow}  \cdots  H\tilde{C}_n (q)\stackrel{R_{n}} {\longrightarrow}  H\tilde{C}_{n+1}(q)\longrightarrow  \cdots  \nonumber\\\nonumber \\\nonumber 
			  H\tilde{B}_{2}(q)  \stackrel{Q_{2}}{\longrightarrow}  H\tilde{B}_{3}(q) \stackrel{Q_{3}} {\longrightarrow}  \cdots  H\tilde{B}_n (q)\stackrel{Q_{n}} {\longrightarrow}  H\tilde{B}_{n+1}(q)\longrightarrow  \cdots  \nonumber 
		\end{eqnarray}	
	are towers of faithful arrows. 	\\  

\end{theorem}
 
\begin{proof}
We write the proof in case $\tilde C$, case $\tilde B$ is identical. Assuming a non trivial dependence relation  $$(\ast) \qquad \sum_w \lambda_w R_n(e_w) =0,$$   we let  
 $ m= \max \{L(w)   \  | \ 
w \in W^c(\tilde C_{n }) \text{ and } \lambda_w \ne 0 \}.$

In the development of $(\ast)$ on the basis  $(g_x)_{x \in W(\tilde{C}_{n+1 } }$, the terms $g_x$ with non-zero coefficient have $L(x)\le m$ and, among those, the ones with $L(x)=m$ come exclusively from
$\sum_{L(w)=m} \lambda_w R_n(e_w)  $. So among those again, using Proposition \ref{coroBC}, the ones with maximal Coxeter length are the $\lambda_w A_w g_{F_n(w)}$ where $L(w)=m$ and $l(w)=a$, 
with $a= \max \{l(w)   \  | \ L(w)=m   \text{ and } \lambda_w \ne 0 \}$. But they are linearly independent, a contradiction.
\end{proof}

\subsection{Faithfulness of towers of $\tilde D$-type Hecke algebras}

For type $\tilde D$, as in type $\tilde A$, where the proof of injectivity was given in 
\cite[Proposition 4.3.3]{Sadek_Thesis},  
 the preceeding proof cannot be used because  the monomorphism $G_n$ defined in  Corollary \ref{Coxeterinj} 
does not transform the  length in the linear way   $F_n$ and $L_n$ do. The proof  of
 \cite[Proposition 4.3.3]{Sadek_Thesis}
or \cite[Proposition 3.3]{Sadek_2017} have to be imitated to get the following Proposition. 
Since  we expect    injectivity to be holding more generally, we don't include the proof.

\begin{proposition}\label{pr_3_3_3}	 
	Let $K= \mathbf Q[q,q^{-1}]$. 			The homomorphism   of algebras  $$P_n:  H\tilde{D}_n (q)   \longrightarrow  H\tilde{D}_{n+1} (q)$$ defined in (\ref{defPn}) is an injection. 		\\  
			\end{proposition}

\section{Full commutativity and  normal forms}\label{notations}

  \subsection{Fully commutative elements}

				In a given Coxeter group $(W,S)$, we know that, from a given reduced expression of an element $w$, we can arrive to any other reduced expression of $w$ only by applying braid relations \cite[\S 1.5 Proposition 5]{Bourbaki_1981}.  Among these relations there are commutation relations: those 
				that  correspond  to   generators $t$ and $s$ in $S$  such that $st$ has order $2$.  \\

	             \begin{definition} Let $\Gamma$ be a Coxeter diagram. 
			Elements of $W(\Gamma$)  for which one can pass from any reduced expression to any other one only by applying commutation relations are called {\rm fully commutative elements}. We denote  by $W^{c}(\Gamma)$, or simply $W^{c}$,  the set of fully commutative elements in $W= W(\Gamma)$. \\
		    \end{definition}
 We now consider  $W(D_{n+1} ) $ generated by $\{ \sigma_{1},  \sigma_{\bar 1}, \dots, \sigma_{n}  \}$. 
The set  $W^{c}(D_{n+1})$     is described 
  by Stembridge in \cite{St}. We use the notation there and the same assumption that the subword  $ \sigma_{\bar 1} \sigma_1$ does not appear (we see it as $\sigma_1 \sigma_{\bar 1}$  for the sake of unicity). For integers $j  \ge i \ge 2$ and $k\ge 1$ let:  		
$$ \begin{aligned}  \      \langle i,j ]    
&= \sigma_i \sigma_{i+1} \dots \sigma_j   \  ;          \    \langle -i, j ]   =  \sigma_i \sigma_{i-1} \dots  \sigma_2  \sigma_1\sigma_{\bar 1}   \sigma_2\dots \sigma_{j-1}  \sigma_j, 
\\
 \langle 1,k]    
&= \sigma_1 \sigma_{2} \dots \sigma_k  \   ; \    \langle -1, k]     = \sigma_{\bar 1}  \sigma_{2} \dots \sigma_k\      \    ;  \    \langle 0, k]   =\sigma_1  \sigma_{\bar 1}  \sigma_{2} \dots \sigma_k  ; 
\end{aligned}  
$$ 
so that $ \langle -1,1]  =  \sigma_{\bar 1}$ and  $ \langle 0,1]  =\sigma_1   \sigma_{\bar 1}$. We also let, for convenience and later use: $  \langle n+1,n ]=1$. 
Then 
every element of $W(D_{n+1})$ 
has a unique reduced expression of the form 
$ \langle m_1,n_1]  \langle m_2,n_2] \dots \langle m_r,n_r]   $ with 
$n\ge n_1 > n_2 > \dots n_r \ge 1$ and $|m_i| \le n_i$ for $1\le i \le r$ 
  ({\it loc.cit.} p.1310).   Now:

\begin{theorem}\label{1_2}{\rm \cite[Theorem 10.1]{St}}
    $W^c(D_{n+1})$ is the set of elements with a reduced expression of the   form 
 \begin{equation}\label{Stembridge}
 \langle m_1,n_1]  \langle m_2,n_2] \dots \langle m_r,n_r]   
  \end{equation}
  with 
$n\ge n_1 > n_2 > \dots n_r \ge 1$ and $|m_i| \le n_i$ for $1\le i \le r$, where the occurrences of 
$ 1 $ and $-1 $ alternate and either  
\begin{enumerate}
\item $ m_1 > \dots >m_s  > |m_{s+1}| = \dots =|m_r| = 1  $ for some  $s \le r$, or 
\item $ m_1 > \dots > m_{r-1}  > -m_r \ge 0 $, $m_{r-1} > 1$, and $m_r \ne -1$. \\
\end{enumerate}
\end{theorem} 
 
	We provide as an example in Appendix \ref{exempleD4} the list of elements of $W^c(D_{4})$. 
	
	 We remark that if $r>1$, then $\langle m_2, n_2 ] \dots \langle m_s, n_s ]$ in (\ref{Stembridge})  above belongs to $W^c( D_ {n})$. From now on we will mean by {\it "the form (\ref{Stembridge})"} the form of a fully commutative element unless we mention otherwise. We notice that  if $ \sigma_{n} $ appears in form (\ref{Stembridge})    above, then it appears necessarily in the first term $\langle m_1,n_1] $ and  either it appears   only once  and we have $n = n_1 \ne -m_1$, or it appears   exactly twice and we have $n = n_1 =-m_1$ and   $r=1$.  Now $ \sigma_{n-1} $ can only appear in the two first terms   $\langle m_1,n_1] \langle m_2,n_2] $ and this happens if and only if $n_1 \ge n-1$.   If   $ \sigma_{n-1} $ appears more than once, then in   form (\ref{Stembridge}) either $n_2 = n-1$ and $n_1 = n $, in which case the constraints of  Theorem \ref{1_2} 
show that $ \sigma_{n-1} $   appears at most twice (for if $m_2=-(n-1)$ then $m_1=n$); or $n_2 < n-1\le n_1$ and $m_1\le -(n-1)$ 
so $r=1$.   
      			
 	\begin{definition}\label{Bex}
				An element $u$ in $W^{c}(D_{n+1})$ is called 
				$\tilde{B}$-{\rm extremal}  if   $ \sigma_{n} $  appears in a (any) reduced expression for $u$. In this case $u$ can be written in  form  (\ref{Stembridge}) with $n_1=n$ and either 
$u= 	 \langle -n, n] $, or $m_1 > -n $ and  $ \sigma_{n} $ appears   only once. 
			\end{definition}

			
			 \begin{definition}\label{Dex}
				An element $u$ in $W^{c}(D_{n+1})$ is called $\tilde{D}$-{\rm extremal}  if   $ \sigma_{n-1} $  appears twice in a (any) reduced expression for $u$. In this case $u$ can be written in  exactly one of the four following   forms: 	
\begin{itemize}   
\item   $\langle-(n-1), n-1]  $ ;
 \item  $ 			 \langle-(n-1), n]   $ ;       
\item $ 			 \langle n, n]  \langle -(n-1), n-1]  $   ;
\item  $		 \langle m_1,n]  \langle m_2,n-1] \dots \langle m_r,n_r] \   $  
 with $r \ge 2$ and  $  m_1 \le n-1$ in addition to the conditions   of Theorem \ref{1_2}.
\end{itemize}
			\end{definition}

		We defined previously the affine length of an element of 	$W (\tilde B_{n+1})$ (Definition \ref{ALB}), which we couldn't do on  
  the full  group $W(\tilde{D}_{n+1}) $ on account of  the   braid relation 
$$\sigma_{n-1}\sigma_{\bar n}\sigma_{n-1} = \sigma_{\bar n} \sigma_{n-1}\sigma_{\bar n}. $$  Since the  	
 words  here  cannot be   subwords of any reduced expression of an element of $W^{c}(\tilde{D}_{n+1})$, it follows (recall \cite[\S 1.5 Proposition 5]{Bourbaki_1981}) that  the number of times $\sigma_{\bar n}$   occurs in a 
 reduced expression of an element of $W^{c}(\tilde{D}_{n+1})$   does not depend of this reduced expression.

			\begin{definition}\label{ALD} Let $u \in W^{c}(\tilde{D}_{n+1})$. 
				We define the {\rm affine length}  of $u$  to be the number of times  $\sigma_{\bar n}$ occurs   in a (any) 
 reduced expression of $u$. We denote  it by $L(u)$. \\
			        \end{definition}	

   \subsection{Reduced expressions of elements of $W^{c}(\tilde{B}_{n+1})$}

\begin{lemma}\label{lemmafull}
Let $w$ be  a fully commutative element in $W(\tilde B_{n+1})$ with $L(w) =m \ge 2$ for $ n\ge 3$ . 
Fix  a reduced expression of $w$ as follows: 
$$
w =  u_1 t_{n+1} u_2 t_{n+1} \dots u_m t_{n+1} u_{m+1}  $$
with $u_i$, for $1\le i \le m+1$, a reduced expression of an element in 
$W^c(D_{n+1})$.  
Then $u_2, \dots, u_m$ are $\tilde B$-extremal  elements and there is a reduced expression of $w$ of the  form:  
 \begin{equation}\label{forme1}
w =  \langle   i_1,n ]  t_{n+1}   \langle   i_2,n ] t_{n+1} \dots   \langle  i_m,n ]  t_{n+1} v_{m+1}  \\
 \end{equation} 
where $ v_{m+1} \in W^c(D_{n+1})$ and one of the following holds: 
\begin{enumerate}
\item $ n+1 \ge i_1 > \dots >i_s  > |i_{s+1}| = \dots =|i_m| = 1  $ for some  $s \le m$, and the occurrences of 
$ 1 $ and $-1 $ alternate,  or 
\item $ n+1 \ge i_1 > \dots > i_{m-1}  > -i_m \ge 0 $, $i_{m-1} > 1$, and $i_m \ne -1$,  or 
\item $ n+1 \ge i_1 \ge   i_2 =\dots = i_m = -n  $ .  
\end{enumerate}
Furthermore, in case (2) we have  $v_{m+1} = 1$. 
 \end{lemma} 
 \begin{proof} Since $t_{n+1}$ commutes with $W( D_{n})$, 
  the fact that the expression is reduced forces  $u_i$ to be $\tilde B$-extremal for $2\le i \le m$. 
We use form~(\ref{Stembridge}) for $u_1$  and  write it as   $u_1=\langle  i_1,n ] x_1$, a reduced expression with $x_1$ in $W^c( D_{n})$ 
and $-n \le i_1 \le n+1$.  Here $x_1$ commutes with $t_{n+1}$ hence   we get a reduced expression 
$$
w =\langle i_1,n ] t_{n+1} x_1 u_2 t_{n+1} \dots u_m t_{n+1} u_{m+1}  .$$ 
Again $x_1u_2 \in W^c(B_{n+1})$  has a reduced expression $ \langle  i_2,n ] x_2$ with $-n \le i_2 \le n $ (since 
$x_1u_2$    is $\tilde B$-extremal)
and $x_2$ in $W^c( D_{n})$, and this $x_2$ commutes with $t_{n+1}$  and can be pushed to the right, leading to 
$$
w =\langle i_1,n ] t_{n+1} \langle i_2,n ] t_{n+1} x_2u_3 t_{n+1} \dots u_m t_{n+1} u_{m+1}  .$$
 Proceeding from left to right we  obtain formally form (\ref{forme1}). 

We now fix $j$, $1\le j < m$.  
\begin{itemize}
\item If $  i_{j+1}  =n$ we must have $i_j = n+1$ (and $j=1$) since, with $i_j \le n$,   our expression would contain  the braid $ \sigma_n t_{n+1}\sigma_n t_{n+1}$, contradicting   the full commutativity; 
\item while  if $ i_{j+1}  =-n$, then if $j=1$, any $i_1$ is possible, but if $j>1$    we must have $  i_{j+1}=i_j=-n$ to avoid the braid $ t_{n+1} \sigma_n t_{n+1}\sigma_n$.  
\item If $1 <| i_{j+1}|<n$ 
  the reduced expression of $u$ contains the subexpression 
$  \langle   i_j,n ]  t_{n+1}  \sigma_{| i_{j+1}|}=\langle   i_j,n ] \sigma_{| i_{j+1}|} t_{n+1}=\dots $        
where   $\sigma_{| i_{j+1}|}$   can be pushed to the left until we reach, if 
$i_j  \le  | i_{j+1}|$  the braid  $\sigma_{| i_{j+1}|}\sigma_{| i_{j+1}|+1}\sigma_{| i_{j+1}|}$, again a contradiction to the full commutativity; hence we have $i_j  >  | i_{j+1}|$. 
\item 
  If $ | i_{j+1}| =1 $ the same argument, pushing  $\sigma_{1}$ or  $\sigma_{\bar 1}$ to the left, past 
$t_{n+1}$ and further,  shows that we must have $i_j>1$ or $i_j=-i_{j+1}$, otherwise we would get  the braid  $\sigma_{1}\sigma_{2}\sigma_{1}$ or 
$\sigma_{\bar 1}\sigma_{2}\sigma_{\bar 1}$. 
\item 
Finally if $  i_{j+1}  =0$, we must have $  i_{j} >1$   to avoid the braids $\sigma_{1}\sigma_{2}\sigma_{1}$ and 
$\sigma_{\bar 1}\sigma_{2}\sigma_{\bar 1}$.
\end{itemize}
Putting together the above rules gives the possibilities  announced and a similar argument provides   the last assertion. 
  \end{proof}

This lemma leads directly to the classification of fully commutative elements in 
$W (\tilde B_{n+1})$.

\begin{theorem}\label{FCB} 
Let $ n\ge 3$. Let $w \in W^c(\tilde B_{n+1})$ with $L(w)   \ge 2$.  
Then $w$ can be written in a unique way as a reduced word of   one and only one of the following two forms, for non negative integers $p$ and $k$:  
\begin{description}
\item[First type]
\begin{equation}\label{formefinalefirsttypeB}
  w =  \langle  -i,n ]   t_{n+1} \left( \langle -n,n ] \;  t_{n+1} \right)^k  \     \langle f,n ]  ^{-1} 
\end{equation}
 with $ -n \le i \le n+1 $ and $-n \le f \le n+1 $.\\

We call such elements {\em quasi-rigid} because their reduced expression is unique up to replacing 
$ \sigma_1 \sigma_{\bar 1}  $ by $ \sigma_{\bar 1}  \sigma_1 $.\\

\item[Second type]
\begin{equation}\label{formefinalesecondtypeB}
\begin{aligned}
w &=   \langle  i_1,n ]  \; t_{n+1}   \dots   \langle  i_p,n ]  \;  t_{n+1}  
  \langle j_1,n ]  \; t_{n+1}  \langle j_2,n ]  \; t_{n+1}\dots \langle j_k,n ]  \; t_{n+1}   
\;    w_r   \quad  \text{if } p>0 ,  \\   
w &=     \langle j_1,n ]  \; t_{n+1}  \langle j_2,n ]  \; t_{n+1}\dots \langle j_k,n ]  \; t_{n+1} 
\;    w_r  \quad   \text{if } p=0, 
\\
w &=   \langle  i_1,n ]  \; t_{n+1}   \langle  i_2,n ]  \; t_{n+1} \dots   \langle  i_p,n ]  \;  t_{n+1}  
\;    w_r   \quad  \text{if } k=0  ,  \\
 \end{aligned}
 \end{equation}
  with $w_r \in W^c(D_{n+1})$   and  \\

\begin{itemize}
\item if $k > 0$:    
\begin{itemize} 
\item for $1\le s \le k $, we have   $j_s= (-1)^{a+s}$ for some integer $a$, that is to say
$j_s=\pm 1$ and  $1$ and $-1$ alternate ;  
\item  and   $w_r= \langle (-1)^{a+k+1},r_1 ] \langle (-1)^{a+k+2},r_2 ] \dots  \langle (-1)^{a+k+u},r_u ]$ 

with $ 1 \le r_u < \dots < r_1  \le n$ ; 
\end{itemize}
\item   if $p > 0$: 
 $ \  
n+1 \ge i_1>... > i_{p-1} > |i_p|     \  $; 
\item   if $p > 0$ and   $  i_p  \le 0$:   $k=0$,  $w_r=1$ and $i_p \ne -n$ ;
\item if $k =0$ and $i_p > 0$:  $w_r$ is of form (\ref{Stembridge}) with $ |m_1 |< i_p$. 
\\

\end{itemize}
\end{description}

The affine length of $w$ of the first (resp. second) type is  $k+1$ (resp. $p + k$) and we have  $ 0 \le p \le n+1  $. \\ 

Now suppose that $L(w) = 1$, then $w$ has a unique reduced expression of the form: 
 \begin{equation}\label{formefinalelongueur1B}
  \langle   i,n ]  \  t_{n+1}   \    v  
  \end{equation}
 where $-n \le  i \le n+1$ and  $v \in W^c(B_{n+1})$ and: 

\begin{itemize}
\item  if $1 < i \le n+1$  then $v$ is of the form (\ref{Stembridge}) such that for $1\leq j \leq r$ either $m_j= n-j+1$ or $|m_j|<i$;
\item   if $i \le 0 $ and  $i \ne -1$ then $v =   \langle h,n ]^{-1}$ with $-n \le h \le n+1$;
\item  if $ i=(-1)^a, $ then either $v = \langle h,n ] ^{-1}$ for $-n \le h \le n+1$, or 

 $v=\langle (-1)^{a+1},r_1 ]\langle (-1)^{a+2},r_2 ] \dots \langle (-1)^{a+u},r_u ]$ with $ 1 \le r_u < \dots < r_1  \le n$.\\
\end{itemize}

Conversely,   every $w$ of the above form is in $W^c(\tilde B_{n+1})$.

\end{theorem} 
\begin{proof}
 We  start   with an  element 
$w \in W^c(\tilde B_{n+1})$ with $L(w) = m \ge 2 $, written  as in (\ref{forme1}). 
Lemma \ref{lemmafull}  provides us with a good part of the Theorem. Indeed we already know that the sequence $i_1, \dots, i_m$ either stops at the first non-positive integer or stabilizes as a sequence of alternating $1$ and $-1$ or as a sequence 
of $-n$, whence the classification into first and second type. So we just have to deal with the final term $w_r$,     with the same kind of arguments as in the proof of Lemma \ref{lemmafull}. 

For an element of first type, any reduced expression of  $w_r$ must start with $\sigma_n$, actually it must be some right truncation of $ \langle -n,n ]$, otherwise we would produce a braid and full commutativity would fail. The easiest way to write it is $(\langle f,n ] )^{-1} $. 

For an element of second type with $k>0$, any reduced expression of  $w_r$ must start with $\sigma_1$ if  $j_k=-1$ or $\sigma_{\bar 1}$  if  $j_k= 1$. Writing 
$w_r= \langle m_1,n_1]  \langle m_2,n_2] \dots \langle m_u, n_u] $ in form (\ref{Stembridge}) 
we see from the conditions of Theorem \ref{1_2} that indeed $m_s=(-1)^{a+k+s}$.

The fact that $w_r=1$ for $p>0$ and $i_p \le 0$ is  stated in Lemma \ref{lemmafull}. Now if $k=0$ and $i_p>0$ we write again $w_r$ in form (\ref{Stembridge}) and check that the only additional condition is $ |m_1 |< i_p$. \\

For an element  $w=\langle  i,n ]  \  t_{n+1}$, $ v    \in W^c(\tilde B_{n+1})$ of affine length   $  1 $, 
     the arguments are similar, using   form (\ref{Stembridge}) for $v$. 
If $i<-1$ or  $i=0$, 
     any reduced expression of    $v $ must start with  $\sigma_n$   on the left and  must be some right truncation of $ \langle -n,n ]$, that is, some $   \langle h,n ]^{-1}$. If $i= \pm 1$, such elements $   \langle h,n ]^{-1}$ are suitable, as well as those in form (6) that start with  $\sigma_1$ if  $i=-1$ or $\sigma_{\bar 1}$  if  $i= 1$, 
whence the result as previously. 

Assume now that  $i> 1$ and write $v= \langle m_1,n_1]  \langle m_2,n_2] \dots \langle m_u, n_u] $ in   form (\ref{Stembridge}). 
If 
$i=n+1$, there is no further condition on $v$, while if $ i\le n$,  any reduced expression of    $v $ must start with  $\sigma_n$ or with $\sigma_t$ with $t <i$.  \\  

The fact that any element of one of these forms is fully commutative is proven by an easy induction. 
\end{proof}

We remark that elements of the first type and elements of affine length $1$ of the form   $\langle i,n ]  \  t_{n+1}    \langle h,n ]^{-1}$ 
  have a unique reduced expression.  Moreover,   an element of affine length at least $2$  has a unique reduced expression if and only if it is of the first type.  
Inserting the elements of affine length   $1$ in the first type and second type sets  would not have given us a partition of the set of those  elements, as we can see in the   example of  $W^c(\tilde B_{4})$ listed in Appendix \ref{exempleB4tilde}. This is the reason  why we handle them separately.

 \begin{remark}{\rm 
We did not include $W^c(\tilde B_{3 })$ in the theorem of classification because   it is clear from the pushing algorithm in the Theorem that $W(\tilde B_{n })$ is an "affinization" of $W( D_{n})$ for $n \ge 4$, which is not the case of $W(\tilde B_{3 })$ which can be looked at as an affinisation of $W( B_{2})$ or $W( A_{3})$. The classification of elements in  $W^c(\tilde B_{3 })$ is fairly easy, we leave it to the reader by recommending to consider our algorithm of pushing while sseing it as an affinisation of  $W( B_{2})$. We choose to list in the Appendix  $W(\tilde B_{4 })$.} 
  \\  
  \end{remark}

\subsection{Reduced expressions of elements of $W^{c}(\tilde{D}_{n+1})$} 		

We call $\psi_1$ (resp. $\psi_n$) the automorphism of  $W(\tilde{D}_{n+1})$ that exchanges  $ \sigma_{\bar 1}$  with $ \sigma_1$  (resp. $ \sigma_{\bar n}$ with $ \sigma_{n}$) and fixes the other generators.
		
\begin{lemma}\label{lemmafullD}
Let $w$ be  a fully commutative element in $W(\tilde D_{n+1})$ with $L(w) =m \ge 2$. 
Fix  a reduced expression of $w$ as follows: 
$$
w =  u_1 \sigma_{\bar n}u_2 \sigma_{\bar n} \dots u_m \sigma_{\bar n} u_{m+1}  $$
with $u_i$, for $1\le i \le m+1$, a reduced expression of a fully commutative element in 
$W^c(D_{n+1})$.  
Then $u_2, \dots, u_m$ are $\tilde D$-extremal  elements, and there is a reduced expression of $w$ of the  form:  
 \begin{equation}\label{formeD}
w =\langle i_1,n]  \langle j_1,n-1] \sigma_{\bar n} \langle i_2,n]  \langle j_2,n-1] \sigma_{\bar n} \dots  \langle i_m,n]  \langle j_m,n-1] \sigma_{\bar n} v_{m+1}  \\
 \end{equation} 
where $ v_{m+1} \in W^c(D_{n+1})$,  $i_1=n+1$ or $|i_1|\le n$, $j_1=n $ or $|j_1|\le n-1$,   $|j_1|< i_1 $
or $ j_1  = - i_1 = \pm 1 $, and for $2 \le t \le m$ the elements 
$\langle i_t,n]  \langle j_t,n-1] $ belong to the following list: 
 \begin{equation}\label{listD} 	
 \begin{aligned}	    
     &(\mathbf a) \quad   \langle n+1 ,n]  \langle -(n-1) ,n-1] = \langle -(n-1), n-1]  ,     \\
				&(\mathbf b) \quad  \langle n, n]  \langle -(n-1), n-1] = \sigma_n 	 \langle -(n-1), n-1], \\  			
 		 &(\mathbf c) \quad  \langle i ,n]  \langle j ,n-1]   \text{ with }  2   \le  i \le n-1   \text{ and }    |j| < i,  \\ 
 		 &(\mathbf d) \quad  \langle i ,n]  \langle j ,n-1]   \text{ with }  i = -j = \pm 1 . 
 \end{aligned}
 \end{equation} 
 \end{lemma} 
 \begin{proof} Since $\sigma_{\bar n}$ commutes with all the generators but $\sigma_{n-1}$, every two consecutive occurrences of  $\sigma_{\bar n}$ must have  in between an occurrence of $\sigma_{n-1}$, for the sake of keeping the expression reduced,  and this occurrence must be properly more than once,  for the sake of full commutativity;  this shows that the $u_i$ with $2\le i \le m$ must be  $\tilde D$-extremal. 

Now we proceed   as in the proof of Lemma \ref{lemmafull}: we use form (\ref{Stembridge}) 
for  $u_1$  and push  the terms $ \langle m_i,n_i]$ with   $n_i < n-1$ to the right of the leftmost 
$\sigma_{\bar n}$ (as well as the term $\sigma_n$ if $n_1=n$ and $n_2<n-1$, so that the term $\langle i_1,n]  \langle j_1,n-1]$, if non trivial, has $|j_1|\le n-1$), then we repeat this procedure with the new $u_2$ thus obtained and the $\sigma_{\bar n}$ appended at its right, and so on, until we get formally   form (\ref{formeD}). 
For the rest it is sufficient to see    Definition \ref{Dex} and notice that, in the $\tilde D$-extremal  element 
$ \langle -(n-1) ,n]  =	  \langle -(n-1), n-1] \sigma_n$, the rightmost term $\sigma_n$ can also be pushed to the right of the $\sigma_{\bar n}$ that follows it on its right, so that this element is not needed in the list. 
  \end{proof}

\begin{lemma}\label{fullandsigmaD} 
Let $w$ be  a fully commutative element in $W(\tilde D_{n+1})$ with $L(w) \!=m \! \ge 2$. We write $w$ in  form  (\ref{formeD}) from   Lemma \ref{lemmafullD}. For   $2 \le t \le m$ 
we  have either $ i_{t}>0$ or $ i_{t}=-1$. Furthermore: 

\begin{enumerate}
\item If    $j_t > 1$   and $  t \le m-1$, then $|i_{t+1}| < j_t < i_t$.  
\item If $-(n-1) <  j_t < -1$ or $j_t=0$,  then   $t=m$ and $v_{m+1} = 1$.
\item If  $j_t= \pm 1$   then   $i_s=-j_t=-j_s$ for all $s> t$, $s \le m$. 
\item If $j_t=-(n-1)$   then $j_{s}=-(n-1)$ for $2 \le s \le m$  and 
$i_{s}=n$ for $2 < s \le m$ while $i_2= n$ or $n+1$.  
 \end{enumerate}
\end{lemma} 

\begin{proof}  We refer to the element  $v_t= \langle i_t,n]  \langle j_t,n-1]$ according to list 
(\ref{listD})  in  Lemma \ref{lemmafullD}. We remark from that list that either $ i_{t}>0$ or $ i_{t}=-1$. 
We call an expression $\sigma_{a_1} \dots \sigma_{a_u}$ {\it admissible} if it is the reduced expression of a fully commutative element. 

In case (1) the element $v_t$  has form (c) with $j >1$. Now  $ \langle i ,n]  \langle j ,n-1]  \sigma_{\bar n} \sigma_s$  is admissible if and only if    $s<j$ or $\sigma_s=  \sigma_{\bar 1}$, since  as before: 
 $\sigma_n$ (resp. $\sigma_{n-1}$, resp.  $\sigma_u$ 
with $j\le u < n-1$)   would produce the braid  $\sigma_n  \sigma_{n-1} \sigma_n$ (resp. $   \sigma_{n-1} \sigma_{\bar n}  \sigma_{n-1}$, resp.   $\sigma_u  \sigma_{u+1} \sigma_u$).  

In case (2) the element  $v_t$ has form (c) with $-(n-1)< j <-1$ or $j=0$. There is no generator $\sigma_s$ that can make $ \langle i ,n] \langle j ,n-1] \sigma_{\bar n} \sigma_s$ admissible. 

In case (3) indeed the element  $v_t$  has form (c) or (d). Again  $\sigma_n \langle -1 ,n-1]  \sigma_{\bar n} \sigma_s$  is admissible if and only if $\sigma_s = \sigma_1$. Applying $\psi_1$ gives the other case. 

In case (4) the element  $v_t$  has form (a) or (b). We note that  $   \langle -(n-1) ,n-1]  \sigma_{\bar n} \sigma_s$  is admissible if and only if    $\sigma_s = \sigma_n$, so the only choice for $v_{t+1}$ is form (b). 
We now look for possible forms of  $v_{t-1}$ using the previous cases :  forms (c) and (d) cannot be followed by form  (a) or (b);  forms (a) and (b) are followed by (b). Hence $v_2$ may have form (a) while $v_s$ for $2 <s \le m$ must have form (b). 
 \end{proof}

These lemmas are the key to the normal form   of fully commutative elements in 
$W (\tilde D_{n+1})$.

\begin{theorem}\label{FCD} 
Let $w \in W^c(\tilde D_{n+1})$ with $L(w)   \ge 2$.  
Then $w$ can be written in a unique way as a reduced word of   one and only one of the following two forms, for non negative integers $p$ and $k$:  

\begin{description}
\item[First type]
\begin{equation}\label{formefinalefirsttypeD}  
 w =  ( \sigma_{\bar n})^\epsilon \sigma_n^\eta  \langle  i,n-1 ] \sigma_{\bar n}(\sigma_n \langle  -(n-1),n-1 ] \;   \sigma_{\bar n} )^k   \  \langle f,n ]^{-1}
\end{equation}
 with $ -n \le f \le n +1 $, $ -(n-1) \le i \le n  $, $\epsilon= 0$ or $1$, $\eta=0$ or $1$, $\epsilon \eta=0$ and if $\epsilon + \eta >0$ then 
$i=-(n-1)$. 
In other words the term $(\sigma_n \langle  -(n-1),n-1 ] \;   \sigma_{\bar n} )^k$ may be preceded either by 
$\sigma_n \langle  -(n-1),n-1 ]\sigma_{\bar n}$ or $ \langle  i,n-1 ]\sigma_{\bar n}$  with $ -(n-1) \le i \le n  $, 
or by  $\sigma_{\bar n} \langle  -(n-1),n-1 ]\sigma_{\bar n}$. 
 \\ 

\item[Second type]
\begin{equation}\label{formefinalesecondtypeD}
\begin{aligned}
w &=  { \langle   i_1,n ]  \langle   j_1,n-1  ]  \; \sigma_{\bar n}   \dots   \langle   i_p,n ]  \langle   j_p,n-1 ]  \;  \sigma_{\bar n}     (  \langle   -1,n ]  \langle   1,n-1 ]   \; \sigma_{\bar n} )^k  
\;    w_r  }  \quad  \text{if } p>0 ,  \\  
w &=       (  \langle   -1,n ]  \langle   1,n-1 ]   \; \sigma_{\bar n} )^k  
\;    w_r  \quad   \text{if } p=0,
 \end{aligned}
 \end{equation}
  with $w_r \in W^c(D_{n+1})$   and 

\begin{itemize}
\item   if $p > 0$: 
 $ \  
n+1 \ge i_1>j_1 >i_2>j_2... > i_{p}> |j_{p}| \ge 0  , i_p > 1, j_p \ne -1  \  $; 
\item   if $p > 0$ and   $j_p  \le 0$ :   $k=0$,  $w_r=1$ and $j_p \ne -(n-1)$;
\item if $k =0$ and $j_p > 1$:  $w_r$ is of form (\ref{Stembridge})   such that $ |m_1 |< j_p$; 
\item if $k > 0$ or if $k=0$ and $j_p=1$:     $w_r= \langle -1,r_1 ] \langle 1,r_2 ] \dots  \langle  (-1)^u,r_u ]$ with $ 1 \le r_u < \dots < r_1  \le n$ and $1,-1$ alternate;   \\
\end{itemize}
or $w$  is the image under $\psi_1 $ of such an element (in this case one must replace  the condition 
$j_p \ne -1$ by $j_p \ne 1$).  \\ 
\end{description}

The affine length of $w$ of the first (resp. second) type is  $k+1 + \epsilon$ (resp. $p + k$) and we have 
 $ 0 \le p \le \frac{n-1}{2}  $. \\ 

Now suppose that $L(w) = 1$, then it has a unique reduced expression of the form: 
 \begin{equation}\label{formefinalelongueur1D}
  \langle i,n]  \langle j,n-1] \sigma_{\bar n}   \    v  \qquad  \text{where } v \in W^c(D_{n+1})  \text{ has form (\ref{Stembridge}) and }
  \end{equation}

\begin{itemize} 
\item either   $  i=n+1 $, $j=n$  and  $v$ is arbitrary,

\item or  
$|j|\le n-1$, $i =-1$ or $ 1 \le i \le n+1$, $|j| < i$ or $j=-i=\pm 1$ , and: 

\begin{itemize}
\item   if $|j|=1$ and $i\le n$ then   $\mathscr{L} (v) \subseteq \{\sigma_{-j}\}$ (with the convention 
$\sigma_{-1} = \sigma_{\bar 1}$);
\item   if $|j|=1$ and $i= n+1$ then   $\mathscr{L} (v) \subseteq \{\sigma_{-j},  \sigma_n \}$;
\item  if $j<-1$ or $j=0$,  and  $i\le n$, then $v=1$; 
\item  if $j<-1$ or $j=0$,  and  $i= n+1$, then $v= \langle f,n ]^{-1} $ with $-n \le f \le n+1 $; 
\item  if $j> 1$ and $i\le n$ then  $ |m_1 |< j $; 
\item if $j> 1$ and $i= n+1$ then    either $ |m_1 |< j $ or, for some $s\le r$, 
$m_1=n_1=n$, \dots, $m_s=n_s=n-s+1$, $|m_{s+1}| <j$.  
\end{itemize}
\end{itemize}

Conversely,   every $w$ of the above form is in $W^c(\tilde D_{n+1})$. 

\end{theorem} 
\begin{proof}
Let $w$ be  a fully commutative element in $W(\tilde D_{n+1})$ with $L(w)  =m  \ge 2$, written in  form  (\ref{formeD}) from   Lemma \ref{lemmafullD}. We consider the terms $ \langle i_t,n]  \langle j_t,n-1]$ 
in $w$ for $2 \le t \le m$. \\

If one of them has form (a) or (b),  
   we know from Lemma \ref{fullandsigmaD} (4)  that   all of them have form (b) except possibly the first one that can have form (a), we thus get  some power of 
$\sigma_n \langle  -(n-1),n-1 ] \;   \sigma_{\bar n} $: this is the first type. 
Same arguments as in the proof of this lemma show that the rightmost term $v_{m+1}$ has to be some right truncation of $\sigma_n \langle  -(n-1),n-1 ]\sigma_n $, 
more easily written in the form $  \langle f,n ]^{-1} $ 
 with   $-n \le f \le n+1 $. We now check the conditions on the introductory term $v_1=\langle i_1,n]  \langle j_1,n-1]$. 

If the  term with $t=2$ has form (b) we have to find conditions for the expression  
 $v_1\sigma_{\bar n}\sigma_n \langle  -(n-1),n-1 ] $ to be admissible. If non trivial, $v_1$ ends with $\sigma_{n-1}$ by the assumption $|j_1|< i_1$ and if it contains a $\sigma_n$ (i.e. $|i_1|\le n$) we must have $j_1=-(n-1)$ to forbid the braid   $\sigma_n\sigma_{n-1}\sigma_n$  (remember that 
$\sigma_{\bar n}$ and $\sigma_n$ commute). So there is no choice but $i_1=n$. We thus have a left truncation of 
$\sigma_n \langle  -(n-1),n-1 ] $. 

If the  term with $t=2$ has form (a) we inspect $v_1\sigma_{\bar n}  \langle  -(n-1),n-1 ] $.  If non trivial, $v_1$ must end with $\sigma_{n}$, which is against our convention in Lemma \ref{lemmafullD}, so $ v_1$ is trivial, hence the form for  the first type. \\  

 Otherwise those terms have form (c) or (d) and we know from Lemma \ref{fullandsigmaD}  (1) that at most the 
$[n/2] $ first terms may have  form (c), so eventually, for $t$ big enough,  all the  terms will have form 
  (d), i.e. 
we get some power of 
$  \langle - 1,n  ]  \langle  1,n-1 ] \;   \sigma_{\bar n}$ or of its image under $\psi_1$: this is the second type. 
We check first the form of   $v_{m+1}$. The proof of Lemma \ref{fullandsigmaD} (3), 
 resp. (2), resp. (1),  gives the case where $k>0$, resp. $k=0$ and $j_p\le0$, resp. $k=0$ and $j_p>0$.    
The form of  the introductory term  
$ \langle   i_1,n ]  \langle   j_1,n-1  ]  \; \sigma_{\bar n}   \dots   \langle   i_p,n ]  \langle   j_p,n-1 ]  \;  \sigma_{\bar n}  $ is an immediate consequence of the previous lemmas.   \\ 

 We have $n+1 \ge i_1 \ge i_p + 2p\ge 2 + 2p$ hence $p \le \frac{n-1}{2}$. \\ 

If $w$ has affine length $1$, as in the proof of Lemma \ref{lemmafullD}, we use the fact that 
$\sigma_{\bar n} $ commutes with all generators except $\sigma_{n-1}$ to obtain the  expression 
 $\langle i,n]  \langle j,n-1] \sigma_{\bar n}       v  $ where either $|j|\le n-1$ or $j=n$ and $i=n+1$. 
The various cases are obtained in the same way as previously, we do not detail this. 
\end{proof}
 
We note that the element of affine length $2$
$$   \sigma_{\bar n} \;    \langle  -(n-1),n-1 ] \;   \sigma_{\bar n}  \  \langle f,n ]^{-1} 
$$   is a first type element.  
Inserting the elements of affine length   $1$ in the first type and second type sets  would not have given us a partition of the set of those  elements  as we can see in the   example in  $W^c(\tilde D_{4})$ given in Appendix \ref{exempleD4tilde}.  This is the reason  why we handle them separately.

\section{The towers of fully commutative elements} 

When a Coxeter group is a parabolic subgroup of another, 
  full commutativity is preserved  by the natural injection. This is not the case for the embeddings  defined in Corollary~\ref{Coxeterinj}:    the monomorphism $G_n: W(\tilde D_{n }) \longrightarrow W(\tilde D_{n+1 })$  does not preserve  full commutativity,   even for the simple reflexion $\sigma_{\bar {n-1}}$, 
while    $L_n: W(\tilde B_{n }) \longrightarrow W(\tilde B_{n+1 })$  preserves  the full commutativity of first type elements and   elements of affine length 1    in $W^c(\tilde B_{n }) $, but  does not preserve it for $t_n\sigma_{n-1}t_n$, for example, in the set of second type fully commutative elements. We will use the normal form for fully commutative elements established in Theorem \ref{FCD} (resp. Theorem \ref{FCB})  to produce embeddings from  $W^c(\tilde B_{n})$ into $W^c(\tilde B_{n+1})$ and from 
 $W^c(\tilde D_{n}) $   into $W^c(\tilde D_{n+1}) $. \\



\subsection{The tower of $\tilde B$-type  fully commutative elements} 
 For $ n>2 $, we denote by   $W^c_1(\tilde B_{n}) $   the set of first type fully commutative elements in addition to fully commutative elements  of affine length 1, and  by  $W^c_2(\tilde B_{n}) $   the set of second type fully commutative elements. We thus  have   the following partition (under the convention $D_3 := A_3$): 
$$
W^c(\tilde B_{n}) = W^c_1(\tilde B_{n}) \bigsqcup W^c_2(\tilde B_{n})  \bigsqcup W^c(D_{n}).
$$

\begin{definition}\label{defIJB}  
For any $w\in W^c(\tilde B_{n })$ we define elements $I(w)$   and  $J(w)$ of $W (\tilde B_{n+1} ) $ by the following expressions: 
\begin{itemize}
\item  if  $w \in W^c_2(\tilde B_{n } )$, then $I(w)$  (resp.  $J(w)$) is  obtained by substituting 
$\sigma_n t_{n+1}$ (resp. $t_{n+1} \sigma_n $)  to $t_n$ in the normal form
(\ref{formefinalesecondtypeB}) for $w$;
\item if  $w \in W^c_1(\tilde B_{n } )$, then $I(w)=J(w)$ is obtained by substituting $\sigma_n t_{n+1}\sigma_n$ to $t_n$ in the normal form (\ref{formefinalefirsttypeB}) or (\ref{formefinalelongueur1B}) for $w$;  
\item if  $w \in W^c(D_{n})$,  then $I(w)=J(w)=w$.
\end{itemize}
\end{definition}

\begin{theorem} \label{IJB} 
For any $w\in W^c(\tilde B_{n })$, the 
   expressions     for $I(w)$ and $J(w)$ in Definition  \ref{defIJB} are reduced and they are reduced expressions for
 fully commutative elements  in $W(\tilde B_{n+1})$. The maps thus defined: 

$$I, J : W^c(\tilde B_{n} ) \longrightarrow W^c(\tilde B_{n+1} ) $$

\medskip\noindent
are injective, preserve the affine length  and satisfy 
  $$ \begin{aligned}
l(I(w)) &=  l(J(w)) = l(w) + L(w) \quad  \text{ for } w \in W^c_2(\tilde B_{n } ),  \\ 
   l(I(w)) &=  l(J(w)) = l(w) +2 L(w)  \quad  \text{ for }  w \in W^c_1(\tilde B_{n } ).
\end{aligned}
$$ 
The injections $I$ and $J$ map first type (resp. second type) elements to  first type (resp. second type) elements 
and their images   intersect exactly on $I(W^c_1(\tilde B_{n})  \bigsqcup W^c(D_{n}))$.  \\ 
\end{theorem}

\begin{proof}
 The proof is mutatis mutandis the proof of  \cite[Theorem 5.2]{Sadek_2017}. We will only emphasize the 
main fact, which must be kept in mind for the applications. For an element $w$ of the second type, the expression of  $J(w)$ is obtained by 
substituting  $t_{n+1} \sigma_n $   to $t_n$ in the normal form for $w$. In order to find the normal form of $J(w)$ 
one uses the property  that $t_{n+1} $ commutes with $\sigma_{\bar 1}$ and $\sigma_i$ for $1 \le i < n$, hence   each $t_{n+1} $ can be pushed to its left until either it is stopped  by a $\sigma_n$ or it is in the leftmost position. In particular,   the leftmost term in the normal form of $J(w)$ is  $t_{n+1}$ whereas no reduced expression of $I(w)$ can have $t_{n+1}$ as its leftmost term (since the expression $ t_{n+1}I(w)$ is also a normal form of the second type). Consequently the images  $I(W^c_2(\tilde B_{n } ))$ and $J(W^c_2(\tilde B_{n } ))$ are disjoint.  
\end{proof} 

\medskip
We remark that the injections $I$ and $J$ on $W^c_1(\tilde C_{n})  \bigsqcup W^c(B_{n})$ are   but the restriction   of $L_n$ in (\ref{Ln}). Actually $I$ and $J$ may be defined on all   $ W(\tilde B_{n }) $, but as we don't need this, we won't examine it further. 

\subsection{The tower of $\tilde D$-type  fully commutative elements} 

 For $ n>3 $, we denote by   $W^c_1(\tilde D_{n}) $   the set of first type fully commutative elements in addition to fully commutative elements  of affine length 1, and  by  $W^c_2(\tilde D_{n}) $   the set of second type fully commutative elements. We thus  have   the following partition: 
$$
W^c(\tilde D_{n}) = W^c_1(\tilde D_{n}) \bigsqcup W^c_2(\tilde D_{n})  \bigsqcup W^c(D_{n}).
$$

\begin{definition}\label{defIJD}  
For any $w\in W^c(\tilde D_{n })$ we define elements $I(w)$   and  $J(w)$ of $W (\tilde D_{n+1} ) $ by the following expressions: 
\begin{itemize}
\item  if  $w \in W^c_2(\tilde D_{n } )$, then $I(w)$  (resp.  $J(w)$) is  obtained by substituting 
$\sigma_{n}\sigma_{n-1}\sigma_{\bar {n}}$ (resp. $\sigma_{\bar {n}}\sigma_{n-1}\sigma_{n}$)  to $\sigma_{\bar {n-1}}$ in the normal form
(\ref{formefinalesecondtypeD}) for $w$;
\item if  $w \in W^c_1(\tilde D_{n } )$, then $I(w)=J(w)$ is obtained by substituting $\sigma_{n-1}\sigma_{\bar {n}}\sigma_{n}$ to $\sigma_{\bar {n-1}}$ in the normal form (\ref{formefinalefirsttypeD}) or (\ref{formefinalelongueur1D}) for $w$, with the following exceptions: 
\begin{itemize}
\item if $w$ is a first type element with $\epsilon=1$ then the leftmost term $\sigma_{\bar {n-1}}$ in the normal form (\ref{formefinalefirsttypeD}) must be substituted by $\sigma_{n}\sigma_{\bar {n}}\sigma_{n-1}$; 
\item if $w$ has affine length $1$ with  $i\ne n$,  then  $\sigma_{\bar {n-1}}$ in   (\ref{formefinalelongueur1D}) must be substituted by $\sigma_{n}\sigma_{n-1}\sigma_{\bar {n}}$. 
\end{itemize}
\item if  $w \in W^c(D_{n})$,  then $I(w)=J(w)=w$.
\end{itemize}
\end{definition}

\begin{theorem} \label{IJD} 
For any $w\in W^c(\tilde D_{n })$, the 
   expressions     for $I(w)$ and $J(w)$ in Definition  \ref{defIJD} are reduced and they are reduced expressions for
 fully commutative elements  in $W(\tilde D_{n+1})$. The maps thus defined: 

$$I, J : W^c(\tilde D_{n} ) \longrightarrow W^c(\tilde D_{n+1} ) $$

\medskip\noindent
are injective, preserve the affine length  and satisfy 
  $$ \begin{aligned}
   l(I(w)) &=  l(J(w)) = l(w) +2 L(w).
\end{aligned}
$$

The injections $I$ and $J$ map first type (resp. second type) elements to  first type (resp. second type) elements 
and their images  intersect exactly on $I(W^c_1(\tilde D_{n})  \bigsqcup W^c(D_{n}))$.  \\ 
\end{theorem}

\begin{proof} We start with $w \in W^c_2(\tilde D_{n } )$ given in form (\ref{formefinalesecondtypeD}) and notice that, since $\sigma_n$ commutes with all generators except $\sigma_{n-1}$, we have : 
$$  I  ( \langle   i_s,n-1 ]  \langle   j_s,n-2  ]  \; \sigma_{\bar{n-1}} )=  \langle   i_s,n  ]  \langle   j_s,n-1  ]  \; \sigma_{\bar{n }}$$
so that $I(w)$ is indeed the reduced expression of an element in $W^c_2(\tilde D_{n+1 } )$ with the same parameters 
$i_1, \dots, i_p, j_1, \dots, j_p, k, w_r$,  as $w$. On the other hand, again using commutation relations, we find: 
$$ 
\begin{aligned}
 J  (& \langle   i_{s-1},n-1 ]   \langle   j_{s-1},n-2  ]  \; \sigma_{\bar{n-1}}
 \langle   i_s,n-1 ]  \langle   j_s,n-2  ]  \; \sigma_{\bar{n-1}} ) \\
&=  
 \langle   i_{s-1},n-1 ] \sigma_{\bar{n }}  \langle   j_{s-1},n  ]  \; 
 \langle   i_s,n-1 ] \sigma_{\bar{n }} \langle   j_s,n   ]  . 
\end{aligned}$$  
Hence $J(w)$ is also the reduced expression  of an element in $W^c_2(\tilde D_{n+1 } )$, 
here with a shift in parameters:  if we denote with primes the parameters for $J(w)$ we have 
$$i^\prime_1= n+1, j^\prime_1= i_1, i^\prime_2= j_1, j^\prime_2= i_2, \dots, j^\prime_p=i_p, 
w_r^\prime =  \langle   x,n  ] w_r,  $$
with $x=j_p$ if $k=0$, or $x = \pm 1$, so $x \ne n+1$. 
In particular this shift acts on the sequence of $1$ and $-1$,  so $J(w)$ will be the image under $\psi_1$ of an element in form (\ref{formefinalesecondtypeD}). 

The injectivity of both maps is clear from the unicity of the normal form,  since the parameters of the images determine the parameters of the source element. Furthermore we remark that in the normal form for $I(w)$, the term on the right of the rightmost $\sigma_{\bar{n }} $ is $w_r $ which belongs to $ W^c (  D_{n } )$ whereas 
for $J(w)$ the corresponding term is  $\langle   x,n  ] w_r $ which contains $\sigma_n$, so  the images of $I$ and $J$ on second type elements do not intersect. 

\medskip  We now turn to first type elements and take 
  $w $ given in form (\ref{formefinalefirsttypeD}). Here  
$$  I  ( \langle   i ,n-2 ]     \; \sigma_{\bar{n-1}}  \sigma_{ {n-1}} \langle  -(n-2) ,n-2 ]  \sigma_{\bar{n-1}})=  
\langle   i ,n-1 ]
\sigma_{\bar {n}}\sigma_{n}\langle  -(n-1) ,n-1 ] \sigma_{\bar {n}} \sigma_n 
$$
 so, if $\epsilon=0$,  $I(w)$ is  the reduced expression of  a first type element,   with the same parameters $\epsilon, \eta, i, k, f$ if $\eta=0$, while if  $\eta=1$ (hence $i=-(n-2)$)  the parameters of $I(w)$, written with a prime, are 
$\epsilon'=\epsilon, \eta'=0, i'=-(n-1), k'=k, f'=f$.   
 If $\epsilon=1$, our element starts with 
 $\sigma_{\bar{n-1}}  \langle  -(n-2),n-2 ]\sigma_{\bar{n-1}} $ and the previous substitution for the leftmost 
 $\sigma_{\bar{n-1}} $ would produce, again for commutation reasons, the braid 
$\sigma_{n}\sigma_{n-1}\sigma_{n}$. We   substitute instead  $\sigma_{\bar {n}}\sigma_{n}\sigma_{n-1}$, getting: 
$$  I  (     \; \sigma_{\bar{n-1}}   \langle  -(n-2) ,n-2 ]  \sigma_{\bar{n-1}} \sigma_{ {n-1}} \cdots)=  
\sigma_{\bar {n}}  \sigma_{n} \langle  -(n-1) ,n-1 ] \sigma_{\bar {n}} \sigma_n \sigma_{ {n-1}} \cdots$$
hence a first type element with parameters  $\epsilon^\prime = \eta^\prime = 0$, 
$i^\prime = n$, $k^\prime = k +1$ and $f^\prime =f$. The set of such parameters is disjoint from the set obtained with $\epsilon =0$, indeed for those we had $i^\prime \le n-1$,   injectivity for first type follows. \\

The last case to consider is elements of affine length $1$. Here substituting $\sigma_{n-1}\sigma_{\bar {n}}\sigma_{n}$ to $\sigma_{\bar {n-1}}$ in the normal form (\ref{formefinalefirsttypeD})  transforms 
$  \langle i,n-1]  \langle j,n-2] \sigma_{\bar {n-1}}  \    v  $ into 
$  \langle i,n-1]  \langle j,n-1] \sigma_{\bar n}  \sigma_{n} \    v $ which is indeed, if $i=n$, the reduced expression 
of a fully commutative element in the required form (\ref{formefinalefirsttypeD}), 
with parameters $i^\prime = n+1$, $j^\prime =j$ and $v^\prime = \sigma_{n}   v $. If $i < n$ the correct substitution is 
$\sigma_{n}\sigma_{n-1}\sigma_{\bar {n}}$, resulting in 
$  \langle i,n ]  \langle j,n-1] \sigma_{\bar n}    \    v $ that again has the required form, with parameters $i^\prime = i$, $j^\prime =j$ and $v^\prime =    v $. The other assertions follow as previously, in particular injectivity follows from the fact that the parameters $i$, $j$ and $v$ determine a fully commutative element of affine length $1$.    
\end{proof}


\section{Faithfulness of towers of  Temperley-Lieb algebras}\label{TL}  

Let $(K,q)$ be as above. Let  $\Gamma$ be a Coxeter graph with associated Coxeter system $(W(\Gamma),S)$ and    Hecke algebra  $H\Gamma(q)$.  Following Graham \cite[Definition 6.1]{Graham}, 
we define the $\Gamma$-type Temperley-Lieb algebra $TL\Gamma(q)$ to be the quotient of the Hecke algebra 
 $H\Gamma(q)$ by the two-sided ideal generated by the elements $ L_{s,t} = \sum_{w\in <s,t>} g_w $, 
where $s$ and $ t  $ are non commuting elements in $S $ such that $st$ has finite order.  For $w$ in $W(\Gamma)$ we denote by $T_w$ the image of $g_w \in H\Gamma(q)$ under the canonical surjection from $H\Gamma(q)$ onto $TL\Gamma(q) $.  The set $\left\{ T_w \ | \ w \in W^c(\Gamma) \right\}$ forms a  $K$-basis for $TL\Gamma(q)$ \cite[Theorem 6.2]{Graham}. \\

	 For $x,y$ in a given ring with  identity,  we define: 
$$\begin{aligned}
V(x,y) &= xyx+xy+yx+x+y+1, \\ 
  Z(x,y)   &=  xyxy+xyx+yxy+xy+yx+x+y+1
.
\end{aligned}
$$

\subsection{Faithfulness of the tower of $\tilde B$-type Temperley-Lieb algebras}\label{TLB}
		
			For $ n\geq 2$, the $\tilde{B}$-type Temperley-Lieb algebra with $n+2$ generators $TL\tilde{B}_{n+1}(q)$ is given by the   set of generators $\left\{ T_{\sigma_{\bar 1}}, T_{\sigma_{1}}, \dots T_{\sigma_{n}}, T_{t_{n+1}}\right\}$, with the defining relations:\\
	 	\begin{equation}\label{definingrelationsB} 	
	   \quad \left\{ \quad  \begin{aligned}
		 &T_{\sigma_{i}} T_{\sigma_{j}} =T_{\sigma_{j}} T_{\sigma_{i}}  \text{ for } 1 \leq i,j\leq n \text{ and } \left| i-j\right| \geq 2, \\
		  &T_{\sigma_{i}} T_{t_{n+1}} =T_{t_{n+1}} T_{\sigma_{i}}   \text{ for  }  1\leq i \leq n-1, \\
	          &T_{\sigma_{i}}T_{\sigma_{i+1}}T_{\sigma_{i}} = T_{\sigma_{i+1}}T_{\sigma_{i}}T_{\sigma_{i+1}}  \text{ for }  1\leq i\leq n-1,\\
    &T_{\sigma_{\bar 1}}T_{\sigma_{2}}T_{\sigma_{\bar 1}}  =T_{\sigma_{2}}  T_{\sigma_{\bar 1}} T_{\sigma_{2}} ,\\
         	    &T_{\sigma_{\bar 1}}T_{t_{n+1}}  =T_{t_{n+1}} T_{\sigma_{\bar 1}}   \text{ and } T_{\sigma_{\bar 1}}T_{\sigma_{i}}  =T_{\sigma_{i}}  T_{\sigma_{\bar 1}}  \text{ for  }  1\leq i \leq n,  i\ne  2  ,\\
			    &T_{t_{n+1}}T_{\sigma_{n}}T_{t_{n+1}}T_{\sigma_{n}} =T_{\sigma_{n}}  T_{t_{n+1}} T_{\sigma_{n}}  T_{t_{n+1}}  ,\\
			    &T^{2}  = (q-1)T  +q    \text{ for }  T \in \left\{ T_{\sigma_{\bar 1}}, T_{\sigma_{1}}, \dots T_{\sigma_{n}}, T_{t_{n+1}}\right\},\\
						  &V(  T_{\sigma_{\bar 1}}, T_{\sigma_{2}} ) = 0 \text{ and } V(T_{\sigma_{i}},T_{\sigma_{i+1}}) = 0  \text{  for }1\leq i\leq n-1, \\
						  &  Z(T_{\sigma_{n}},T_{t_{n+1}})= 0  .
			  \end{aligned} \right.     \qquad 
		\end{equation}
 
 We set $TL\tilde{B}_{2}(q)= K$. We temporarily denote by 
$h_w$, $w \in W^c(\tilde{B}_n)$, the basis elements of  $TL\tilde{B}_{n}(q)$ to distinguish them from those of $TL\tilde{B}_{n+1}(q)$.

		                     
                    \begin{lemma}\label{morphismFnB} 
The morphism of algebras $Q_n: H\tilde{B}_n (q)   \longrightarrow   H\tilde{B}_{n+1} (q) $ defined 
in (\ref{defQn})   induces the following morphism of algebras, which we also denote by  $Q_n$:	
			\begin{eqnarray}
				Q_{n}: TL\tilde{B}_{n}(q) &\longrightarrow& TL\tilde{B}_{n+1}(q) \nonumber\\
				h_{\sigma_{i}} &\longmapsto & T_{\sigma_{i}} \  \   \text{ for } 1 \leq i \leq n-1 \nonumber\\
h_{\sigma_{\bar 1}} &\longmapsto & T_{\sigma_{\bar 1}}   \nonumber\\
				h_{t_{n}} &\longmapsto & T_{\sigma_{n}} T_{t_{n+1}} T^{-1}_{\sigma_{n}}. \nonumber
			\end{eqnarray}
The restriction of $Q_n$ to $  TLD_{n }(q)$ is an injective morphism into $TLD_{n+1}(q)$  and satisfies $Q_n(h_w)= g_{I(w)}= g_{J(w)}$  for  $w \in W^c( D_{n})$. 
                    \end{lemma}
                    
                    \begin{proof} See  \cite[Lemma 6.1]{Sadek_2017}.
                    \end{proof}
		 
		 The aim of this section is to show, using  the normal form  of Theorem \ref{FCB},   that the morphism $Q_n$ is an  injection.  
We  set $p=1/q$, so that we have: 
\begin{equation}\label{basicB} 
 Q_n(h_{t_{n}}) =   p T_{\sigma_{n}t_{n+1}\sigma_{n}}+ 
  (p-1) T_{\sigma_{n}t_{n+1}}. 
\end{equation}

\begin{lemma}\label{formulaB} 
 
If  $ w \in W^c_1(\tilde B_{n}) \bigsqcup W^c(D_{n}) $ we have: 
$$
Q_n(h_w) =    p^{L(w)}   T_{I(w)} 
+   \sum_{\begin{smallmatrix}L(x)\le L(w)\\ l(x)<l(I(w))\end{smallmatrix}}  \alpha_x T_x  \qquad (\alpha_x \in K). 
$$

\end{lemma}

\begin{proof} Same as \cite[Lemma 6.2]{Sadek_2017}. 
   \end{proof} 

 \begin{proposition}\label{formula2B}
  Let $w$ be in $W^c_2(\tilde B_{n}) $, then for some $\alpha_x, \beta_y  \in K$ we have
 $$
Q_n(h_w) =   (-1)^{L(w)}   T_{I(w)} +  (-p)^{L(w)}   T_{J(w)} +  \sum_{\begin{smallmatrix}L(y)=L(w)\\ l(y)<l(I(w))\end{smallmatrix}}  \beta_y T_y 
+  \sum_{L(x)< L(I(w))}  \alpha_x T_x.
$$
 \end{proposition}

\begin{proof} Here also the proof is exactly the same as for type $\tilde C$ in \cite[Proposition 6.3]{Sadek_2017}. Indeed   the proof in {\it loc.cit.} relies on the relations involving 
$T_{\sigma_{n-1}}$, $T_{\sigma_{n}}$ and  $T_{t_{n+1}}$, that are the same in $TL\tilde{B}_{n+1}(q)$  and 
$TL\tilde{C}_{n+1}(q) $, and on the fact that  $T_{t_{n+1}}$ commutes with  $R_n( TLB_{n }(q))$, 
here replaced by 
$Q_n( TLD_{n }(q))$. We mention the following more precise statement from {\it loc.cit.}:   \\  

{\it   Let $w$ be in $W^c_2(\tilde B_{n}) $, whose normal form  (\ref{formefinalesecondtypeB})  begins and ends with $t_n$. Then for some $\alpha_x, \beta_y  \in K$ we have: }  
\begin{equation}\label{preciseB} 
\begin{aligned} 
Q_n(h_w) =   (-1)^{L(w)}   &T_{I(w)} +  (-p)^{L(w)}   T_{J(w)} 
 \\ &+  T_{t_{n+1}}T_{\sigma_{n}}T_{t_{n+1}} \   (\sum_{\begin{smallmatrix}L(y)=L(w)-2\\ l(y)<l(I(w))-3 \end{smallmatrix}}  \beta_y T_y  )
+  \sum_{L(x)< L(I(w))}  \alpha_x T_x.
\end{aligned} 
\end{equation}    
\end{proof}

 We can conclude as in \cite[Theorem 6.4]{Sadek_2017}, with the same proof: 

\begin{theorem}\label{RB} 
The tower of affine Temperley-Lieb  algebras

\begin{eqnarray}
			 TL\tilde{B}_{2}(q) \stackrel{Q_{2}} {\longrightarrow} TL\tilde{B}_{3}(q) \longrightarrow \cdots  TL\tilde{B}_{n}(q)\stackrel{Q_{n}} {\longrightarrow}  TL\tilde{B}_{n+1}(q)\longrightarrow  \cdots  \nonumber\\\nonumber
		\end{eqnarray}
	
	is a tower of faithful arrows. 	\\  
\end{theorem}

		
		\subsection{Faithfulness of the tower of $\tilde D$-type Temperley-Lieb algebras}\label{TLD}
		
			For $ n\geq 3$, the $\tilde{D}$-type Temperley-Lieb algebra with $n+2$ generators $TL\tilde{D}_{n+1}(q)$ is given by the   set of generators $\left\{ T_{\sigma_{\bar 1}}, T_{\sigma_{1}}, \dots T_{\sigma_{n}}, T_{\sigma_{\bar n}}\right\}$, with the defining relations:\\
	 	\begin{equation}\label{definingrelationsD} 	
	   \quad \left\{ \quad  \begin{aligned}
		 &T_{\sigma_{i}} T_{\sigma_{j}} =T_{\sigma_{j}} T_{\sigma_{i}}  \text{ for } 1 \leq i,j\leq n \text{ and } \left| i-j\right| \geq 2, \\
   &T_{\sigma_{\bar 1}}T_{\sigma_{\bar n}}  =T_{\sigma_{\bar n}} T_{\sigma_{\bar 1}}   \text{ and } T_{\sigma_{\bar 1}}T_{\sigma_{i}}  =T_{\sigma_{i}}  T_{\sigma_{\bar 1}}  \text{ for  }  1\leq i \leq n,  i\ne  2  ,\\
   &  T_{\sigma_{\bar n}}T_{\sigma_{i}}  =T_{\sigma_{i}}  T_{\sigma_{\bar n}}  \text{ for  }  1\leq i \leq n,  i\ne  n-1  ,\\
	          &T_{\sigma_{i}}T_{\sigma_{i+1}}T_{\sigma_{i}} = T_{\sigma_{i+1}}T_{\sigma_{i}}T_{\sigma_{i+1}}  \text{ for }  1\leq i\leq n-1,\\
    &T_{\sigma_{\bar 1}}T_{\sigma_{2}}T_{\sigma_{\bar 1}}  =T_{\sigma_{2}}  T_{\sigma_{\bar 1}} T_{\sigma_{2}}  \text{ and }  T_{\sigma_{\bar n}}T_{\sigma_{n-1}}T_{\sigma_{\bar n}}  =T_{\sigma_{n-1}}  T_{\sigma_{\bar n}} T_{\sigma_{n-1}} ,\\
			    &T^{2}  = (q-1)T  +q    \   \text{ for }  T \in \left\{ T_{\sigma_{\bar 1}}, T_{\sigma_{1}}, \dots , T_{\sigma_{n}}, T_{\sigma_{\bar n}}\right\},\\
&V(  T_{\sigma_{\bar 1}}, T_{\sigma_{2}} )=V(  T_{\sigma_{\bar n}}, T_{\sigma_{n-1}} ) = 0,  \\
						  &  V(T_{\sigma_{i}},T_{\sigma_{i+1}}) = 0  \text{  for }1\leq i\leq n-1     .
			  \end{aligned} \right.     \qquad 
		\end{equation}
 
 We set $TL\tilde{D}_{3}(q)= K$. In the following we denote by 
$h_w$, $w \in W^c(\tilde{D}_n)$, the basis elements of  $TL\tilde{D}_{n}(q)$ to distinguish them momentarily from those of $TL\tilde{D}_{n+1}(q)$.

		                     
                    \begin{lemma}\label{morphismFnD} 
The morphism of algebras $P_n: H\tilde{D}_n (q)   \longrightarrow   H\tilde{D}_{n+1} (q) $ defined 
in (\ref{defPn})   induces the following morphism of algebras, which we also denote by  $P_n$:	
			\begin{eqnarray}
				P_{n}: TL\tilde{D}_{n}(q) &\longrightarrow& TL\tilde{D}_{n+1}(q) \nonumber\\
				h_{\sigma_{i}} &\longmapsto & T_{\sigma_{i}} \  \   \text{ for }    \sigma_{i} \in \left\{  \sigma_{\bar 1} ,  \sigma_{1} , \dots  ,\sigma_{n-1}  \right\}, \nonumber\\
				h_{\sigma_{\bar {n-1}}} &\longmapsto & T_{\sigma_{n}} T_{\sigma_{n-1}} T_{\sigma_{\bar n}} T^{-1}_{\sigma_{n-1}} T^{-1}_{\sigma_{n}}. \nonumber
			\end{eqnarray}
The restriction of $P_n$ to $  TLD_{n }(q)$ is an injective morphism into $TLD_{n+1}(q)$  and satisfies $P_n(h_w)= g_{I(w)}= g_{J(w)}$  for  $w \in W^c( D_{n})$. 
                    \end{lemma}
                    
                    \begin{proof}
                    
                    The lemma follows after noticing that
                    \begin{eqnarray}
                     V(P_n(h_{\sigma_{\bar {n-1}}}),P_n(h_{\sigma_{n-2}})) = (T_{\sigma_{n}} T^{-1}_{\sigma_{\bar n}})  V (T_{\sigma_{n-1}},T_{\sigma_{n-2}})(T_{\sigma_{n}} T^{-1}_{\sigma_{\bar n}})^{-1}.\nonumber
                    \end{eqnarray}
                    \end{proof}
		 
		 The aim of this section is to show, using  the normal form  of Theorem \ref{FCD},   that the morphism $P_n$ is an  injection.  
We  set $p=1/q$. We will use repeatedly \\
\begin{equation}\label{basicD} 
\begin{aligned}
 P_n(h_{\sigma_{\bar {n-1}}})  =  &T_{\sigma_{n}\sigma_{n-1}\sigma_{\bar n}} + p T_{\sigma_{n-1}\sigma_{\bar n}\sigma_{n}}+ pT_{\sigma_{n}\sigma_{\bar n}\sigma_{n-1}} +p^{2}T_{\sigma_{\bar n}\sigma_{n-1}\sigma_{n}}  \\ 
&+ pT_{\sigma_{n}\sigma_{n-1}} +p^{2}T_{\sigma_{n-1}\sigma_{n}} +pT_{\sigma_{n-1}\sigma_{\bar n}} +p^{2}T_{\sigma_{\bar n}\sigma_{n-1}} + (p^{2}+p) T_{\sigma_{\bar n}\sigma_{n}}  \\ 
&+p^{2}T_{\sigma_{n}} +p^{2}T_{\sigma_{n-1}} +p^{2}T_{\sigma_{\bar n}} + (p^{2}-p)  
\end{aligned}
\end{equation}
as well as those two rules that follow from   the defining relations (\ref{definingrelationsD}):  \\

\noindent 
(i) 
In $TL\tilde{D}_{n+1}(q)$, a product $T_w T_y$, $w, y \in W^c(\tilde{D}_{n+1})$,   is either equal to 
$T_{wy}$, if $l(wy)=l(w)+l(y)$ and $wy$ is fully commutative, or equal to a linear combination of terms $T_z$, $z\in W^c(\tilde{D}_{n+1})$,  with 
$L(z) \le L(w)+L(y)$  and $l(z) < l(w)+l(y)$. \\

\noindent 
(ii) 
 When   a braid $T_{\sigma_{i}}T_{\sigma_{i+1}}T_{\sigma_{i}}$ appears in a computation, the use of 
$V(T_{\sigma_{i}},T_{\sigma_{i+1}})=0$ replaces it by a sum of terms $T_z$  with  $l(z) =2$, $1$ or $0$, hence the length decreases.  \\

\begin{proposition}\label{formulaLge2D} 
Let $w \in W^c(\tilde D_{n}) $ be of affine length at least $2$. Then   for some $\alpha_x, \beta_y  \in K$ we have, 
if $w$ is a first type element: 
$$
P_n(h_w) =    p^{L(w)}   T_{I(w)} 
+   \sum_{\begin{smallmatrix}L(x)\le L(w)\\ l(x)<l(I(w))\end{smallmatrix}}  \alpha_x T_x  \qquad (\alpha_x \in K), 
$$
and if $w$ is a second type element: 
 $$
P_n(h_w) =     T_{I(w)} +   p^{2 L(w)}   T_{J(w)} +  \sum_{\begin{smallmatrix}L(y)=L(w)\\ l(y)<l(I(w))\end{smallmatrix}}  \beta_y T_y 
+  \sum_{L(x)< L(I(w))}  \alpha_x T_x.
$$
\end{proposition}

\begin{proof} We go back to Lemma \ref{lemmafullD} and write $w$ as in (\ref{formeD}): 
$$
w =\langle i_1,n-1]  \langle j_1,n-2] \sigma_{\bar{n-1}}   \dots  \langle i_m,n-1]  \langle j_m,n-2] \sigma_{\bar{n-1}} v_{m+1}   
$$ 
with the notation there, in particular, for $2 \le t \le m$, the element  
$\langle i_t,n-1]  \langle j_t,n-2] $ belongs to    list (\ref{listD}): 
$$
 \begin{aligned}	    
     &(\mathbf a) \quad   \langle n  ,n-1]  \langle -(n-2) ,n-2] = \langle -(n-2), n-2]  ,   \nonumber\\
				&(\mathbf b) \quad  \langle n-1, n-1]  \langle -(n-2), n-2] = \sigma_{n-1} 	 \langle -(n-2), n-2], \nonumber\\  			
 		 &(\mathbf c) \quad  \langle i ,n-1]  \langle j ,n-2]   \text{ with }  2   \le  i \le n-2   \text{ and }    |j| < i, \nonumber\\ 
 		 &(\mathbf d) \quad  \langle i ,n-1]  \langle j ,n-2]   \text{ with }  i = -j = \pm 1 . \nonumber
 \end{aligned}
$$
We examine the subword  $x=  \sigma_{\bar{n-1}} \langle i_t,n-1]  \langle j_t,n-2] \sigma_{\bar{n-1}}$. When developing 
$P_n(h_w)$ we will   develop  $P_n(h_x)$ as 
$ P_n(h_x) =  P_n(h_{\sigma_{\bar {n-1}}})  T_{\langle i_t,n-1]  \langle j_t,n-2]} P_n(h_{\sigma_{\bar {n-1}}})  $ 
where $P_n(h_{\sigma_{\bar {n-1}}}) $ is given by (\ref{basicD}). We are interested in terms of maximal affine length $L(w)$ and maximal Coxeter length $l(w)$, so we keep only in (\ref{basicD}) the  four terms of Coxeter length $3$. We have to check the following products:

 \begin{figure}[ht]
				\centering
				\begin{tikzpicture}
               \begin{scope}[xscale = 1]

  \node at (-4.5,1.5)  {$[1]~  \ \     T_{\sigma_{n}\sigma_{n-1}\sigma_{\bar n}} $};
  \node at (-4.5,0.5)  {$[2]~p  \    T_{\sigma_{n-1}\sigma_{\bar n}\sigma_{n}}$};
  \node at (-4.5,-0.5)  {$[3]~p  \     T_{\sigma_{n}\sigma_{\bar n}\sigma_{n-1}}$};
  \node at (-4.5,-1.5)  {$[4]~p^{2} \    T_{\sigma_{\bar n}\sigma_{n-1}\sigma_{n}}$};

\draw (-1.5,0)  -- (-2.8,-1.4);
\draw (-1.5,0)  -- (-2.8,-0.4);
\draw (-1.5,0)  -- (-2.8,0.4);
\draw (-1.5,0)  -- (-2.8,1.4);
	
  \node at (0,0)  {$  T_{\langle i_t,n-1]  \langle j_t,n-2]} $};

\draw (1.5,0)  -- (2.8,-1.4);
\draw (1.5,0)  -- (2.8,-0.4);
\draw (1.5,0)  -- (2.8,0.4);
\draw (1.5,0)  -- (2.8,1.4);

  \node at (4.5,1.5)  {$[1']~ \  \   T_{\sigma_{n}\sigma_{n-1}\sigma_{\bar n}} $  };
  \node at (4.5,0.5)  {$[2']~p \  T_{\sigma_{n-1}\sigma_{\bar n}\sigma_{n}}$  };
  \node at (4.5,-0.5)  {$[3']~p \     T_{\sigma_{n}\sigma_{\bar n}\sigma_{n-1}}$  };
  \node at (4.5,-1.5)  {$[4']~p^{2} \   T_{\sigma_{\bar n}\sigma_{n-1}\sigma_{n}}$  };

\end{scope}
               \end{tikzpicture}
							
			\end{figure}	
Each time a resulting subword is either non reduced or non fully commutative (i.e. containing a braid), it will contribute  only to terms of strictly shorter length. We thus eliminate by inspection all possible combinations except the following : 
1c1', 1d1', 2b2', 3a2', 4c4', 4d4'. 
 
By definition, if $w$ is a second type element, then our middle term has form (c) or (d) and we will get the maximal length elements by replacing the occurrences of   $P_n(h_{\sigma_{\bar {n-1}}})$
either all  by $ T_{\sigma_{n}\sigma_{n-1}\sigma_{\bar n}} $, giving rise to a term $T_{I(w)}$, or all by 
$p^{2}   T_{\sigma_{\bar n}\sigma_{n-1}\sigma_{n}}$, giving rise to a term $p^{2L(w)} T_{J(w)}$.

If $w$ is a first type element, a middle term has form (a) or (b),  but form (a) can only occur if $t=2$ by Lemma \ref{fullandsigmaD}. We get the maximal length elements by replacing all occurrences of   $P_n(h_{\sigma_{\bar {n-1}}})$
   by $ p       T_{\sigma_{n-1}\sigma_{\bar n}\sigma_{n}} $, except, if form (a) occurs (for $t=2$), the leftmost one, that must be replaced by $p     T_{\sigma_{n}\sigma_{\bar n}\sigma_{n-1}}$. We indeed get 
$p^{ L(w)} T_{I(w)}$ as the only leading term. 
\end{proof} 
 
	\medskip
\begin{lemma}\label{formulaD} 
Let $w \in W^c(\tilde D_{n}) $ be 
of affine length   $1$  and write 
as in (\ref{formefinalelongueur1D}): 
$$w=  \langle i,n-1]  \langle j,n-2] \sigma_{\bar{n-1}}   \    v   \quad ( v \in   W^c(D_{n})) .$$ 
Let  $\nu=1$ if $i=n$ and $\nu=0$ otherwise. If $i=n$, or if $i < n$ and $ v \notin   W^c(D_{n-1}) $, 
 we have: 
$$
P_n(h_w) =    p^\nu  T_{I(w)} + \sum_{\begin{smallmatrix}L(x)=1\\ l(x)=l(I(w))\\ 
x \notin \text{ Im } I \end{smallmatrix}} 
 \alpha_x T_x 
+   \sum_{\begin{smallmatrix}L(x)\le 1\\ l(x)<l(I(w))\end{smallmatrix}}  \alpha_x T_x  \qquad (\alpha_x \in K). 
$$

If $i < n$ and $ v \in   W^c(D_{n-1}) $, we have: 
$$
P_n(h_w) =      T_{I(w)} + p T_{I(\bar w)} + \sum_{\begin{smallmatrix}L(x)=1\\ l(x)=l(I(w))\\ 
x \notin \text{ Im } I \end{smallmatrix}} 
 \alpha_x T_x 
+   \sum_{\begin{smallmatrix}L(x)\le 1\\ l(x)<l(I(w))\end{smallmatrix}}  \alpha_x T_x  \qquad (\alpha_x \in K)  
$$
where $\bar w =  \langle i,n-2 ] \sigma_{\bar {n-1}}  \langle j,n-1]    v  $.     
\end{lemma}

\begin{proof} We recall, 
from the proof of Theorem   
\ref{IJD}, that 
the image of $I$ consists of elements
\begin{equation}\label{imageI} 
 \langle k',n ]  \langle l',n-1] \sigma_{\bar{n }}   \    u'    \quad \text{   with  } 
 \left\{
\begin{matrix}
&\text{either }   k'=n+1, l'=l, u'=\sigma_{n} u  \text{ if } k=n , 
\\ 
&\text{ or }   k'=k, l'=l, u'=  u  \text{ if } k<n , 
\end{matrix} 
\right. 
\end{equation}
where 
$(k, l, u) $ are the parameters that determine uniquely a fully commutative element  of affine length $1$ and form (\ref{formefinalelongueur1D}).  In particular elements in the image of $I$ satisfy: 
$$\left\{\begin{aligned}
 &-(n-1) \le k' \le n-1 \text{ or }   k'=n+1 ,  \\
&-(n-2) \le l' \le n-1 .
\end{aligned}\right.
$$ 

Let   $w$ have affine length $1$ and write 
as in (\ref{formefinalelongueur1D}): 
$w=  \langle i,n-1]  \langle j,n-2] \sigma_{\bar{n-1}}   \    v  $ with $ v \in   W^c(D_{n}) $. 
We have  $$P_n(h_w)=  T_{\langle i ,n-1]  \langle j ,n-2]} P_n(h_{\sigma_{\bar {n-1}}}) T_v $$ 
where $P_n(h_{\sigma_{\bar {n-1}}}) $ is given by (\ref{basicD}),  from which we keep the four terms of Coxeter  length $3$. As previously we have to check the  products in the figure below to identify   
$ T_{I(w)}$ or $  p   T_{I(w)}$,  and  possible  terms among the others  of the form 
$\lambda T_x$, $\lambda \in K$ and $l(x)=l(I(w))$, and  study whether   $x$   belongs to the image of $I$. 

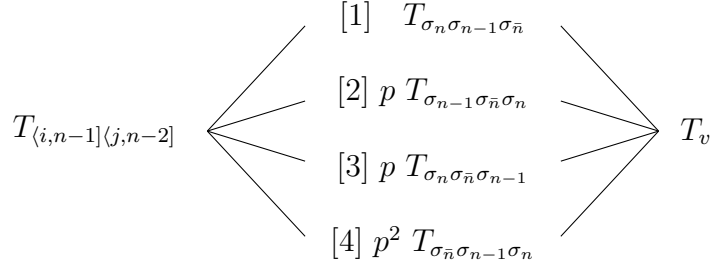
 \begin{figure}[ht]
				\centering
				\begin{tikzpicture}
               \begin{scope}[xscale = 1]

  \node at (-4.5,0)  {$  T_{\langle i,n-1]  \langle j,n-2]} $};

\draw (-3,0)  -- (-1.7,-1.4);
\draw (-3,0)  -- (-1.7,-0.4);
\draw (-3,0)  -- (-1.7,0.4);
\draw (-3,0)  -- (-1.7,1.4);

  \node at (0,1.5)  {$[1]~  \ \     T_{\sigma_{n}\sigma_{n-1}\sigma_{\bar n}} $};
  \node at (0,0.5)  {$[2]~p  \    T_{\sigma_{n-1}\sigma_{\bar n}\sigma_{n}}$};
  \node at (0,-0.5)  {$[3]~p  \     T_{\sigma_{n}\sigma_{\bar n}\sigma_{n-1}}$};
  \node at (0,-1.5)  {$[4]~p^{2} \    T_{\sigma_{\bar n}\sigma_{n-1}\sigma_{n}}$};

\draw (3,0)  -- (1.7,-1.4);
\draw (3,0)  -- (1.7,-0.4);
\draw (3,0)  -- (1.7,0.4);
\draw (3,0)  -- (1.7,1.4);
	
  \node at (3.5,0)  {$  T_v $};

\end{scope}
               \end{tikzpicture}
							
			\end{figure}	
We write the resulting products in the form  $\langle k',n ]  \langle l',n-1] \sigma_{\bar{n }}   \    u'   $ and abbreviate fully commutative as fc. 
\begin{enumerate}
\item  If  $i=n$,   term [2] does   provide a $ p T_{I(w)}$. The other three terms may provide some $T_x$ 
with $l(x)=l(I(w))$ (for 
  instance,  if $j = n-1$ and $v=1$, the four terms have this property: this is  (\ref{basicD})), but we will show that  such $x$ do not belong  to the image of $I$. 

For term [1], we  observe that 
$\langle j,n-2] \sigma_{n}\sigma_{n-1}\sigma_{\bar n} v =  \langle n,n ] \langle j,n-1]   \sigma_{\bar n} v  $ 
 cannot  belong to the image of $I$ since $k'=n$, impossible from the remark following (\ref{imageI}). 
Similarly 
term [3] would correspond to 
$\langle j,n-2] \sigma_{n}\sigma_{\bar n} \sigma_{n-1} v = \sigma_{\bar n}  \sigma_{n} \langle j,n-1]     v  $
that 
 cannot  belong to the image of $I$ since $l'=n$, and 
  term [4]   would correspond to 
$\langle j,n-2]\sigma_{\bar n} \sigma_{n-1} \sigma_{n} v = \sigma_{\bar n}  \langle j,n ]      v  $ that 
 cannot  belong to the image of $I$ for the same reason.    
The claim is proved in case $i=n$. \\

\item  If  $i< n$,    term   [1]     provides a $   T_{I(w)}$. Again we must check   the other three terms. 

Term [2] with $j\ne -(n-2) $   gives the braid  $\sigma_{n-1} \sigma_{n-2}\sigma_{n-1}$, while for 
$j = -(n-2) $ (hence $i=n-1$) we get
$\langle -(n-1),n-1]  \sigma_{\bar n} \sigma_{n} v      $ which is not in the image of $I$ because 
$l'= -(n-1)$.

Term [4] gives 
$\langle i,n-1] \langle j,n-2]\sigma_{\bar n} \sigma_{n-1} \sigma_{n} v   = 
\langle i,n-1] \sigma_{\bar n}  \langle j,n ]    v    $, so $k'=n+1$, $l'=i$, $u'=\langle j,n ]    v $. 
For this element to be in the image of $I$ and of maximal length it is necessary that   $u'=\langle j,n ]    v $ could be written as 
$\sigma_{n} u$,  reduced form of a fc element, with $\langle i,n-2] \sigma_{\bar{n-1}}  u    $ satisfying the conditions in  (\ref{formefinalelongueur1D}). This is impossible because $j < n$.

Finally term [3]  gives 
$   \langle i,n-1 ] \langle j,n-2]  \sigma_{n} \sigma_{\bar n} \sigma_{n-1} v=  \langle i,n-1 ] \sigma_{\bar n}\sigma_{n} \langle j,n-1]    v  $, 
so $k'=n+1$, $l'=i$, $u'=\sigma_{n} \langle j,n-1]    v $. This element might be the image under $I$ of $ \langle i,n-2 ] \sigma_{\bar {n-1}}  \langle j,n-1]    v  $. 
Here $v$ belongs to $W^c(D_n)$ and we see directly from (\ref{formefinalelongueur1D}) and Definition 
\ref{Bex} that if $\sigma_{n-1}$ appears in $v$, then $\langle j,n-1]    v  $ cannot be reduced fc since $j  < n-1$. 
On the other hand, if $v$ belongs to $W^c(D_{n-1})$,   conditions  (\ref{formefinalelongueur1D}) for $w$ 
imply conditions  (\ref{formefinalelongueur1D}) for $ \bar w =\langle i,n-2 ] \sigma_{\bar {n-1}}  \langle j,n-1]    v  $, therefore in this case the element  $p  T_{I(\bar w)}$, where  $I(\bar w)=  \langle i,n-1 ] \sigma_{\bar n}\sigma_{n} \langle j,n-1]    v  $   satisfies  
  $l(I(\bar w))=l(I(w))$, appears in $R_n(h_w)$. 
 \end{enumerate}
The lemma follows.
  \end{proof}

 \medskip

\begin{theorem}\label{RD} 
The tower of affine Temperley-Lieb  algebras

\begin{eqnarray}
			  TL\tilde{D}_{3}(q)  \stackrel{P_{3}}{\longrightarrow}  TL\tilde{D}_{4}(q) \stackrel{P_{4}} {\longrightarrow}  \cdots  TL\tilde{D}_{n}(q)\stackrel{P_{n}} {\longrightarrow}  TL\tilde{D}_{n+1}(q)\longrightarrow  \cdots  \nonumber\\\nonumber
		\end{eqnarray}
	
	is a tower of faithful arrows. 	\\  
\end{theorem}

\begin{proof} We need to show that $P_n$ is an injective homomorphism of algebras. 
A basis for $TL\tilde{D}_{n}(q)$ is given by the elements $h_w$ where $w$ runs over 
$ W^c(\tilde D_{n })$. Assume there are non trivial dependence relations between the images of these basis elements. Pick one such relation, say 
\begin{equation}\label{eqn:drD}
\sum_w \lambda_w P_n(h_w) =0 , 
\end{equation}
 and let 
$m= \max \{ L(w)   \  | \ 
w \in W^c(\tilde D_{n }) \text{ and } \lambda_w \ne 0 \}$.

If $m=0$ or $m\ge 2$ the proof goes as in \cite[Theorem 6.4]{Sadek_2017}, using respectively Lemma \ref{morphismFnD} 
and Proposition \ref{formulaLge2D}. 

 If $m=1$ we let:
$$ \begin{aligned} X &= \{   w \in W^c(\tilde D_{n })/ L(w)=1 \text{ and } \lambda_w \ne 0 \} ,  \\ 
 l &= \max  \{ l(w) / w \in X \} , \\ 
 X_0 &=    \{   w \in X / l(w)=l \}.
 \end{aligned} 
$$
 We 
 use Lemma  
 \ref{formulaD} to develop  the terms in    (\ref{eqn:drD}) with $\lambda_w \ne 0 $. 
We obtain terms $T_z$ where $z$ has affine length $0$ or $1$ 
and we want to set apart those $T_z$ where $z$ has affine length   $1$ and maximal Coxeter length. They are 
exactly the terms $T_z$ where $L(z)=1$ and $l(z)= l+2$  in the development, using  Lemma  
 \ref{formulaD}, of the  sum 
$$\sum_{w \in X_0}  \lambda_w P_n(h_w).    
$$
Among those terms we have those $T_z$ with $z \in \text{Im }I$ and those $T_z$ with $z \notin \text{Im }I$, and the spaces they generate are in direct sum. So the following sum must be equal to $0$: 
\begin{equation}\label{eqn:last}
\sum_{w \in X_0}  \lambda_w  \left( p^{\nu_w}  T_{I(w)} + \delta_w \  p  T_{I(\bar w)} \right)
\end{equation}
where we have set, with the notation in  Lemma  
 \ref{formulaD}:  $\delta_w = 1$ if  $i < n$ and $ v \in   W^c(D_{n-1}) $, and $\delta_w = 0$ (and $\bar w=w $) 
otherwise. 

If this sum does not contain any $w$ with  $\delta_w = 1$, the injectivity of $I$ shows directly that  $
 \lambda_w =0$ for $w \in X_0$, a contradiction. 
Otherwise we recall from 
(\ref{imageI})  that the image by $I$ of terms of affine length $1$ is the disjoint union of two subsets, one is the image $A$ of those terms with $i=n$, the other  the image $B$ of those terms with $i<n$. We notice that 
if  $\delta_w = 1$, then $I(w) \in B$ and   $I(\bar w) \in A$. The linear combination of the terms $T_x$ with 
$x \in B$ in (\ref{eqn:last}) has to be zero, and it is equal to $$ 
\sum_{w \in X_0, I(w) \in B}  \lambda_w   T_{I(w)},   
$$   so the coefficients $
 \lambda_w  $, for $w \in X_0$ and $\delta_w = 1$, must be $0$, again a contradiction.  
\end{proof} 

\bigskip 

 \clearpage

\appendix


\section{Elements in $W^{c}(D_{4} )$}\label{exempleD4}

  We   list the elements in $W^{c}(D_{4} )$ given in their normal form.  $W^{c}(D_{4})$ has $48$ elements, see \cite[Proposition 10.4 and Appendix Table 2]{St}:\\

	 $1, \sigma_{1},\sigma_{2}, \sigma_{3}, \sigma_{1} \sigma_{2},  \sigma_{2}\sigma_{1},   \sigma_{2}\sigma_{3}, \sigma_{3}\sigma_{2},  \sigma_{3}\sigma_{1}, \sigma_{3}\sigma_{1}\sigma_{2}, \sigma_{3}\sigma_{2}\sigma_{1}, \sigma_{2}\sigma_{3}\sigma_{1}, \sigma_{1}\sigma_{2}\sigma_{3}, \sigma_{2}\sigma_{3}\sigma_{1}\sigma_{2}, $\\      
	
	  $\sigma_{\bar 1},  \sigma_{\bar 1} \sigma_{2},  \sigma_{2}\sigma_{\bar 1},  \sigma_{3}\sigma_{\bar 1}, \sigma_{3}\sigma_{\bar 1}\sigma_{2}, \sigma_{3}\sigma_{2}\sigma_{\bar 1}, \sigma_{2}\sigma_{3}\sigma_{\bar 1}, \sigma_{\bar 1}\sigma_{2}\sigma_{3}, \sigma_{2}\sigma_{3}\sigma_{\bar 1}\sigma_{2},$\\
 
 $\sigma_{1} \sigma_{\bar 1} , \sigma_{3}\sigma_{1} \sigma_{\bar 1},\sigma_{1} \sigma_{\bar 1}\sigma_{2}, \sigma_{1} \sigma_{2}\sigma_{\bar 1}, \sigma_{2}\sigma_{1} \sigma_{\bar 1}, \sigma_{\bar 1}\sigma_{2} \sigma_{1}, \sigma_{2}\sigma_{1} \sigma_{\bar 1}\sigma_{2}, \sigma_{3}\sigma_{2}\sigma_{1} \sigma_{\bar 1}\sigma_{2}, \sigma_{2}\sigma_{1} \sigma_{\bar 1}\sigma_{2}\sigma_{3}, $\\	

$\sigma_{1}\sigma_{\bar 1}\sigma_{2}\sigma_{3}, \sigma_{1} \sigma_{2}\sigma_{3}\sigma_{\bar 1}, \sigma_{\bar 1} \sigma_{2}\sigma_{3}\sigma_{1},  \sigma_{3}\sigma_{\bar 1} \sigma_{1}\sigma_{2}, \sigma_{2} \sigma_{3}\sigma_{1}\sigma_{\bar 1}, \sigma_{3} \sigma_{2}\sigma_{1}\sigma_{\bar 1} , \sigma_{3} \sigma_{1}\sigma_{2}\sigma_{\bar 1},  \sigma_{3} \sigma_{\bar 1}\sigma_{2}\sigma_{1}$\\
 
 $ \sigma_{2} \sigma_{3}\sigma_{1} \sigma_{2}\sigma_{\bar 1}, \sigma_{2}\sigma_{3}\sigma_{\bar 1}\sigma_{2} \sigma_{1}, \sigma_{1}  \sigma_{2}\sigma_{3}\sigma_{\bar 1} \sigma_{2}, \sigma_{\bar 1} \sigma_{2}\sigma_{3}\sigma_{1} \sigma_{2},  \sigma_{2}\sigma_{3}\sigma_{1}\sigma_{\bar 1} \sigma_{2},$\\

$ \sigma_{1}\sigma_{2}\sigma_{3} \sigma_{\bar 1}\sigma_{2}\sigma_{1},\sigma_{\bar 1} \sigma_{2}\sigma_{3}\sigma_{1}\sigma_{2} \sigma_{\bar 1} , \sigma_{3}\sigma_{2}\sigma_{1} \sigma_{\bar 1}\sigma_{2}\sigma_{3} $.\\
	
	\clearpage

\section{Elements in  $W^{c}(\tilde{B}_{4} )$ of positive affine length.}\label{exempleB4tilde}

  We   list the elements in $W^{c}(\tilde{B}_{4} )$ of positive affine length.  
  \begin{itemize}
\item   First  type elements:  
$$   c \  t_{4}(\sigma_{3}\sigma_{2}\sigma_{1}\sigma_{\bar 1}\sigma_{2}\sigma_{3} \;  t_{4} )^h \  d \qquad \text{with } \     \left\{\begin{matrix}   h \ge 1,  \cr 
    \text{if } c=1, (\text{ say } p=0)  \text{then }h\ge 2 . \cr
    c  \in  1, \sigma_{3}, \sigma_{2}\sigma_{3}, \sigma_{1}\sigma_{2}\sigma_{3}, \sigma_{\bar 1}\sigma_{2}\sigma_{3}, \sigma_{1}\sigma_{\bar 1}\sigma_{2}\sigma_{3}, \cr
   \sigma_{2}\sigma_{1}\sigma_{\bar 1}\sigma_{2}\sigma_{3}, \sigma_{3}\sigma_{2}\sigma_{1}\sigma_{\bar 1}\sigma_{2}\sigma_{3},\cr
   d  \in   1, \sigma_{3}, \sigma_{3}\sigma_{2}, \sigma_{3}\sigma_{2}\sigma_{1}, \sigma_{3}\sigma_{2}\sigma_{\bar 1}, \sigma_{3}\sigma_{2}\sigma_{1}\sigma_{\bar 1},\cr
    \sigma_{3}\sigma_{2}\sigma_{1}\sigma_{\bar 1}\sigma_{2}, \sigma_{3}\sigma_{2}\sigma_{1}\sigma_{\bar 1}\sigma_{2}\sigma_{3},\cr 
  \end{matrix}\right.
 $$
 
\item   Second type elements:   
$$   a \  ( \sigma_{i_k}\sigma_{2}\sigma_{3} \; t_{4}  )^k  \  b    \   \qquad \text{with } \     \left\{\begin{matrix} k \ge 1,  i_1= \bar 1 \cr 
   \text{if } a=1 , \text{then }k\ge 2, \cr     
\text{if }a \in  t_4,  \sigma_{3}t_4, \sigma_{2}\sigma_{3}t_4 , \sigma_{1}\sigma_{2}\sigma_{3}t_4 \text{ then } k\ge 1,\cr 
 \text{if } k=1 \text{ then } b \in 1, \sigma_{1}, \sigma_{3}, \sigma_{1}\sigma_{2}, \sigma_{3}\sigma_{2}, \sigma_{3}\sigma_{1},\cr \sigma_{3}\sigma_{1}\sigma_{2}, \sigma_{3}\sigma_{2}\sigma_{1}, \sigma_{1}\sigma_{2}\sigma_{3}, \sigma_{3}\sigma_{2}\sigma_{\bar 1}, \sigma_{1}\sigma_{2}\sigma_{\bar 1},\cr
  \sigma_{1} \sigma_{2} \sigma_{3} \sigma_{\bar 1},  \sigma_{3} \sigma_{2} \sigma_{1} \sigma_{\bar 1},  \sigma_{3} \sigma_{1} \sigma_{2} \sigma_{\bar 1},\cr
   \sigma_{3} \sigma_{2} \sigma_{1} \sigma_{\bar 1} \sigma_{2},  \sigma_{1} \sigma_{2} \sigma_{3} \sigma_{\bar 1} \sigma_{2},\cr
    \sigma_{1} \sigma_{2} \sigma_{3} \sigma_{\bar 1} \sigma_{2} \sigma_{1}, \sigma_{3} \sigma_{2} \sigma_{1} \sigma_{\bar 1} \sigma_{2}  \sigma_{3},\cr
     \text{if }a\in \sigma_{\bar 3}\sigma_{2}\sigma_{3}\sigma_{\bar 1}\sigma_{2}\sigma_{\bar 3},  \sigma_{1}\sigma_{\bar 3}\sigma_{2}\sigma_{3}\sigma_{\bar 1}\sigma_{2}\sigma_{\bar 3}, \text{then } k\ge 0,\cr  
   \text{here if  } k=0  \text{ then } b \in  1,  \sigma_1,  \sigma_1\sigma_{2},  \sigma_1\sigma_{2}\sigma_{\bar 1},  \sigma_1\sigma_{2}\sigma_{3}\sigma_{\bar 1},  \cr
    \sigma_1\sigma_{2}\sigma_{3}\sigma_{\bar 1},  \sigma_1\sigma_{2}\sigma_{3}\sigma_{\bar 1}\sigma_{2},  \sigma_1\sigma_{2}\sigma_{3}\sigma_{\bar 1}\sigma_{2}\sigma_{1}. \cr  
    \text{if } a\in  t_4 \sigma_{1} \sigma_{\bar 1} \sigma_{2}  \sigma_{3}t_4,  \sigma_{3}t_4 \sigma_{1} \sigma_{\bar 1} \sigma_{2}  \sigma_{3}t_4,  \sigma_{2} \sigma_{3}t_4\sigma_{1} \sigma_{\bar 1} \sigma_{2}  \sigma_{3}t_4,\cr 
    t_4\sigma_{2} \sigma_{1} \sigma_{\bar 1} \sigma_{2}  \sigma_{3}t_4,  \sigma_{3}t_4 \sigma_{2}\sigma_{1} \sigma_{\bar 1} \sigma_{2}  \sigma_{3}t_4,\text{ then } k=0 \text{ and }  b=1,\cr
    \text{if } a\in  t_4 \sigma_{2}  \sigma_{3}t_4,  \sigma_{3}t_4  \sigma_{2}  \sigma_{3}t_4, \cr
    \text{ and } k=0 \text{ then } b \in  1,  \sigma_1,  \sigma_1\sigma_{2},  \sigma_1\sigma_{2}\sigma_{3}, \sigma_{\bar 1},  \sigma_{\bar 1}\sigma_{2},  \sigma_{\bar 1}\sigma_{2}\sigma_{3}, \cr 
    \sigma_{1}\sigma_{\bar 1}, \sigma_{1}\sigma_{\bar 1}\sigma_{2}, \sigma_{\bar 1}\sigma_{2}\sigma_{1}, \sigma_{1}\sigma_{2}\sigma_{\bar 1}, \sigma_1\sigma_{\bar 1}\sigma_{2}\sigma_{3},  \sigma_1\sigma_{2}\sigma_{3}\sigma_{\bar 1},\cr
    \sigma_{\bar 1}\sigma_2\sigma_{3}\sigma_{1}, \sigma_1\sigma_{2}\sigma_{3}\sigma_{\bar 1}\sigma_{2}, \sigma_{\bar 1}\sigma_{2}\sigma_{3}\sigma_{1}\sigma_{2},\cr
   \sigma_1\sigma_{2}\sigma_{3}\sigma_{\bar 1}\sigma_{2}\sigma_{1},  \sigma_{\bar 1}\sigma_{2}\sigma_{3}\sigma_{1}\sigma_{2}\sigma_{\bar 1}.\end{matrix}\right.
 $$

In addition to $\psi_1 (w)$ in this type.\\

\item  Elements of affine length $1$:   
$$  \qquad   e t_4 f   \quad \text{with } \     \left\{\begin{matrix} 
\text{either }  e \in 1,  \sigma_3 \text{ and }   f \in  W( D_{4 }), \qquad \qquad 
 \cr 
 \text{either }  e = \sigma_{2}\sigma_3 \text{ and }   f \in  W( D_{4 }) - \sigma_{2}, \sigma_{3}\sigma_{2}, \sigma_{1}\sigma_{2}, \sigma_{2}\sigma_{3}\sigma_{1},\sigma_{2}\sigma_{3}\sigma_{2} \sigma_{1}, \cr
  \sigma_{2}\sigma_{\bar 1}, \sigma_{2}\sigma_{3}\sigma_{\bar 1}, \sigma_{2}\sigma_{3}\sigma_{\bar 1}\sigma_{2}, \sigma_{2}\sigma_{1}\sigma_{\bar 1}, \sigma_{2}\sigma_{1}\sigma_{\bar 1}\sigma_{2},\sigma_{2}\sigma_{1}\sigma_{\bar 1}\sigma_{2} \sigma_{3},\cr
   \sigma_{2}\sigma_{3}\sigma_{1}\sigma_{\bar 1}, \sigma_{2}\sigma_{3}\sigma_{1}\sigma_{2}\sigma_{\bar 1}, \sigma_{2}\sigma_{3}\sigma_{\bar 1}\sigma_{2}\sigma_{1}, \sigma_{2}\sigma_{3}\sigma_{1}\sigma_{\bar 1}\sigma_{2},\cr 
   \text{either }  e \in \sigma_{1}\sigma_{2}\sigma_{3}, [  \sigma_{\bar 1}\sigma_{2}\sigma_{3} ]\text{ and }   f \in 1, \sigma_{3}, \sigma_{3}\sigma_{2}, \sigma_{3}\sigma_{2}\sigma_{1}, \cr
    \sigma_{\bar 1}, \sigma_{\bar 1}\sigma_{2}, \sigma_{3}\sigma_{\bar 1}, \sigma_{3}\sigma_{\bar 1}\sigma_{2}, \sigma_{3}\sigma_{2}\sigma_{\bar 1}, \sigma_{\bar 1}\sigma_{2}\sigma_{3}, \cr
    \sigma_{\bar 1}\sigma_{2}\sigma_{1}, \sigma_{\bar 1}\sigma_{2}\sigma_{3}\sigma_{1}, \sigma_{3}\sigma_{2}\sigma_{1}\sigma_{\bar 1}, \sigma_{3}\sigma_{\bar 1}\sigma_{2}\sigma_{1}, \sigma_{3}\sigma_{2}\sigma_{1}\sigma_{\bar 1}\sigma_{2},\cr
    \sigma_{\bar 1}\sigma_{2}\sigma_{3}\sigma_{1}\sigma_{2}, \sigma_{\bar 1}\sigma_{2}\sigma_{3}\sigma_{1}\sigma_{2}\sigma_{\bar 1},\sigma_{3}\sigma_{2}\sigma_{1} \sigma_{\bar 1}\sigma_{2}\sigma_{3}, \text{ and } \psi_1  [f],   \qquad \qquad 
 \cr
\text{ or } e \in  \sigma_{1} \sigma_{\bar 1}\sigma_{2}\sigma_{3}, \sigma_{2}\sigma_{1} \sigma_{\bar 1}\sigma_{2}\sigma_{3}, \sigma_{3}\sigma_{2}\sigma_{1} \sigma_{\bar 1}\sigma_{2}\sigma_{3} \text{ and }  f \in  1, \sigma_{3}\sigma_{2},\cr
 \sigma_{3}\sigma_{2}\sigma_{1}, \sigma_{3}\sigma_{2} \sigma_{\bar 1}, \sigma_{3}\sigma_{2}\sigma_{1} \sigma_{\bar 1}, \sigma_{3}\sigma_{2}\sigma_{1} \sigma_{\bar 1}\sigma_{2}, \sigma_{3}\sigma_{2}\sigma_{1} \sigma_{\bar 1}\sigma_{2}\sigma_{3}. \\
\end{matrix}\right.$$
  \end{itemize}

{\rm Notice that if $h$ and $k$ were allowed to be null,  i,.e one of the two types was allowed to contain elements of affine length equals to one, then  $\sigma_1\sigma_2\sigma_3 t_4 \sigma_3$ could be obtained in two different ways:  $ a= \sigma_1\sigma_2\sigma_3 t_4, b=\sigma_3$ with  $k=0$, and $ c = \sigma_1\sigma_2\sigma_3, d= \sigma_3 $ with $h=0$.}
 
    \clearpage
\section{Elements in $W^{c}(\tilde{D}_{4} )$ of positive affine length.}\label{exempleD4tilde}  

  We   list the elements in $W^{c}(\tilde{D}_{4} )$ of positive affine length. 
  \begin{itemize}
\item   First  type elements:  
$$   
\begin{aligned}
& (\sigma_{\bar 3})^\epsilon (c= \langle   i,n-1 ] )  \sigma_{\bar 3}(\sigma_3  \sigma_{2} \sigma_{1} \sigma_{\bar 1} \sigma_{2} \;   \sigma_{\bar 3} )^h  \  d \qquad \text{with } \     \\  
&\left\{\begin{matrix}  
    c  \in  \left\{ 1, \sigma_{2},  \sigma_{1} \sigma_{2},  \sigma_{\bar 1} \sigma_{2},  \sigma_{1} \sigma_{\bar 1} \sigma_{2},\sigma_{2} \sigma_{1} \sigma_{\bar 1} \sigma_{2}\right\}  ,  \cr
 \epsilon =1 \text{iff }   c= \sigma_{2} \sigma_{1} \sigma_{\bar 1} \sigma_{2}, \text{here } h \ge0\cr
    \text{if } c=1 , \text{then }h\ge 2, \text{(with }q p=0 \text{\bf ) }  \cr
   d  \in  1,  \sigma_{3}, \sigma_{3}\sigma_{2}, \sigma_{3}\sigma_{2}\sigma_{1}, \sigma_{3}\sigma_{2}\sigma_{\bar 1},\sigma_{3} \sigma_{2}\sigma_{1}\sigma_{\bar 1} , \cr
   \sigma_{3}\sigma_{2}\sigma_{1} \sigma_{\bar 1}\sigma_{2}, \sigma_{3}\sigma_{2}\sigma_{1} \sigma_{\bar 1}\sigma_{2}\sigma_{3}. 
\end{matrix}\right.
\end{aligned}
 $$
 
\item   Second type elements:   
$$   a \   ( \sigma_{1} \sigma_{2} \sigma_{3}  \sigma_{\bar 1}\sigma_{2}   \; \sigma_{\bar 3} )^k  \  b    \   \qquad \text{with } \     \left\{\begin{matrix} \text{if } a=1 , \text{then }k\ge 2, \cr 
    \text{if }a \in  \sigma_{2}\sigma_{\bar 3} , \sigma_{3}\sigma_{2}\sigma_{\bar 3}, \cr 
    \sigma_{\bar 1} \sigma_{2}\sigma_{\bar 3} , \sigma_{3}\sigma_{\bar 1}\sigma_{2}\sigma_{\bar 3} ,  \sigma_{2}\sigma_{3}\sigma_{\bar 1}\sigma_{2}\sigma_{\bar 3},  \text{then } k\ge 1,\cr 
     \text{if }a\in \sigma_{\bar 3}\sigma_{2}\sigma_{3}\sigma_{\bar 1}\sigma_{2}\sigma_{\bar 3},  \sigma_{1}\sigma_{\bar 3}\sigma_{2}\sigma_{3}\sigma_{\bar 1}\sigma_{2}\sigma_{\bar 3}, \text{then } k\ge 0,\cr  
   \text{if } k\ge 0  \text{ then } b \in  1,  \sigma_1,  \sigma_1\sigma_{2},  \sigma_1\sigma_{2}\sigma_{\bar 1},  \sigma_1\sigma_{2}\sigma_{3}\sigma_{\bar 1},  \cr
    \sigma_1\sigma_{2}\sigma_{3}\sigma_{\bar 1},  \sigma_1\sigma_{2}\sigma_{3}\sigma_{\bar 1}\sigma_{2},  \sigma_1\sigma_{2}\sigma_{3}\sigma_{\bar 1}\sigma_{2}\sigma_{1}.   
  \end{matrix}\right.
 $$

In addition to $\psi_1 (w)$ in this type.\\

\item  Elements of affine length $1$:   
$$  \qquad   e \sigma_{\bar 3}f   \quad \text{with } \     \left\{\begin{matrix} 
\text{either }  e = 1 \text{ and }   f \in W^{c}(D_{4} ), \qquad \qquad 
 \cr 
 \text{either }  e =  \sigma_{2}\text{ and }   f \in W^{c}(D_{4} )-  \sigma_{2},  \sigma_{2} \sigma_{1},  \sigma_{2} \sigma_{3}, \sigma_{2} \sigma_{3} \sigma_{1},   \cr 
 \sigma_{2} \sigma_{3} \sigma_{1} \sigma_{2},  \sigma_{2} \sigma_{\bar 1}, \sigma_{2} \sigma_{3} \sigma_{\bar 1},  \sigma_{2} \sigma_{3} \sigma_{\bar 1} \sigma_{2},  
  \sigma_{2} \sigma_{1} \sigma_{\bar 1},  \cr 
  \sigma_{2} \sigma_{1} \sigma_{\bar 1} \sigma_{2},  \sigma_{2} \sigma_{1} \sigma_{\bar 1} \sigma_{2} \sigma_{3},  
  \sigma_{2} \sigma_{3} \sigma_{1} \sigma_{\bar 1},  \cr  
  \sigma_{2} \sigma_{3} \sigma_{1} \sigma_{2} \sigma_{\bar 1}, \sigma_{2} \sigma_{3} \sigma_{\bar 1} \sigma_{2} \sigma_{1}, \sigma_{2} \sigma_{3} \sigma_{1} \sigma_{\bar 1} \sigma_{2}\qquad \qquad 
 \cr 
 \text{either }  e \in  \sigma_{1}\sigma_{2},  [ \sigma_{\bar 1} \sigma_{2}] \text{ and }   f \in 1, \sigma_{3}, \sigma_{3} \sigma_{2},  \sigma_{3}\sigma_{2}\sigma_{1}, \sigma_{\bar 1}\cr
  \sigma_{\bar 1} \sigma_{2},   \sigma_{3}\sigma_{\bar 1},   \sigma_{3} \sigma_{\bar 1}\sigma_{2},   \sigma_{3} \sigma_{2}\sigma_{\bar 1}, \sigma_{\bar 1}\sigma_{2}\sigma_{3},  \sigma_{\bar 1}\sigma_{2}\sigma_{1}, \cr
    \sigma_{3}\sigma_{2}\sigma_{1}\sigma_{\bar 1}\sigma_{2},  \sigma_{\bar 1}\sigma_{2}\sigma_{3}\sigma_{1},  \sigma_{3}\sigma_{2}\sigma_{1}\sigma_{\bar 1},  \sigma_{3}\sigma_{\bar 1}\sigma_{2}\sigma_{1}\cr
    \sigma_{\bar 1}\sigma_{2}\sigma_{3}\sigma_{1}\sigma_{2}\sigma_{\bar 1},  \sigma_{3}\sigma_{2}\sigma_{1}\sigma_{\bar 1}\sigma_{2}\sigma_{3},  \text{ and } [ \psi_1 (f)], \qquad \qquad 
 \cr 
 \text{either }  e =  \sigma_{3} \sigma_{2}\text{ and }   f \in 1, \sigma_{1},  \sigma_{1}\sigma_{2}\sigma_{3}, \sigma_{\bar 1}\cr
  \sigma_{\bar 1} \sigma_{2},   \sigma_{\bar 1}\sigma_{2}\sigma_{3},   \sigma_{1} \sigma_{\bar 1},   \sigma_{1} \sigma_{\bar 1}\sigma_{2},\sigma_{1} \sigma_{2}\sigma_{\bar 1},  \sigma_{\bar 1}\sigma_{2}\sigma_{1}, \cr
    \sigma_{1}\sigma_{\bar 1}\sigma_{2}\sigma_{3},  \sigma_{1}\sigma_{2}\sigma_{3}\sigma_{\bar 1},  \sigma_{\bar 1}\sigma_{2}\sigma_{3}\sigma_{1},  \sigma_{1}\sigma_{2}\sigma_{3}\sigma_{\bar 1}\sigma_{2}\cr
    \sigma_{\bar 1}\sigma_{2}\sigma_{3}\sigma_{1}\sigma_{2},  \sigma_{1}\sigma_{2}\sigma_{3}\sigma_{\bar 1}\sigma_{2}\sigma_{1},  \sigma_{\bar 1}\sigma_{2}\sigma_{3}\sigma_{1}\sigma_{2}\sigma_{\bar 1}, \qquad \qquad 
 \cr 
 \text{either }  e \in  \sigma_{3} \sigma_{1}\sigma_{2},  \sigma_{2}\sigma_{3} \sigma_{1}\sigma_{2},  [\sigma_{3} \sigma_{\bar 1}\sigma_{2},  \sigma_{2}\sigma_{3} \sigma_{\bar 1}\sigma_{2}]\text{ and }   f \in 1,  \sigma_{\bar 1}\cr
  \sigma_{\bar 1} \sigma_{2},   \sigma_{\bar 1}\sigma_{2}\sigma_{3},   \sigma_{\bar 1}\sigma_{2}\sigma_{1}    \sigma_{\bar 1}\sigma_{2}\sigma_{3}\sigma_{1}, \cr
    \sigma_{\bar 1}\sigma_{2}\sigma_{3}\sigma_{1}\sigma_{2},  \sigma_{\bar 1}\sigma_{2}\sigma_{3}\sigma_{1}\sigma_{2}\sigma_{\bar 1},  \text{ and }   [\psi_1 (f)], \qquad \qquad 
 \cr 
 \text{either }  e \in  \sigma_{\bar 1}\sigma_{1} \sigma_{2},  \sigma_{2}\sigma_{\bar 1}\sigma_{1} \sigma_{2},   \sigma_{3}\sigma_{\bar 1}\sigma_{1} \sigma_{2},  \sigma_{3} \sigma_{2}\sigma_{\bar 1}\sigma_{1} \sigma_{2}\text{ and }   f \in 1, \cr
 \sigma_{3}, \sigma_{3} \sigma_{2},  \sigma_{3}\sigma_{2}\sigma_{1}, \sigma_{3} \sigma_{2}\sigma_{\bar 1},\cr
      \sigma_{3}\sigma_{2}\sigma_{1}\sigma_{\bar 1}\sigma_{2}, \sigma_{3}\sigma_{2}\sigma_{1}\sigma_{\bar 1}, 
    \sigma_{3}\sigma_{2}\sigma_{1}\sigma_{\bar 1}\sigma_{2}\sigma_{3}, \qquad \qquad 
 \cr 
 \text{either }  e \in \sigma_{1} \sigma_{2}\sigma_{3}\sigma_{\bar 1} \sigma_{2},  [\sigma_{\bar 1} \sigma_{2}\sigma_{3}\sigma_{1} \sigma_{2}]\text{ and }   f \in 1, \cr
 \sigma_{1} , \sigma_{1}\sigma_{2}, \sigma_{1}\sigma_{2}\sigma_{3}, \sigma_{1}\sigma_{2}\sigma_{\bar 1}, \sigma_{1}\sigma_{2}\sigma_{3}\sigma_{\bar 1}, \cr
  \sigma_{1}\sigma_{2}\sigma_{3}\sigma_{\bar 1}\sigma_{2}, \sigma_{1}\sigma_{2}\sigma_{3}\sigma_{\bar 1}\sigma_{2}\sigma_{1}, \text{ and }   [\psi_1 (f)], \qquad \qquad 
 \cr 
\text{ or } e =  \sigma_{2}\sigma_{3} \sigma_{1}\sigma_{\bar 1} \sigma_{2} \text{ and }  f =1  . 
\end{matrix}\right.$$
  \end{itemize}

{\rm Notice that if $h$ and $k$ were allowed to be null, thus if the first and second types where allowed to be of affine length smaller than two then we would loose the unicity of the form. For example  then  $ \sigma_{2}\sigma_{\bar 3}$ could be obtained in two different ways:  $ a=  \sigma_{2}\sigma_{\bar 3}, b= 1$, $k=0$, and 
$ c =  \sigma_{2} , d= 1 $, $h=0$.}

		\vspace{1cm}

 
\renewcommand{\refname}{REFERENCES}

\end{document}